\documentclass[english]{amsart}
\usepackage{amsmath,amssymb,amscd,amsfonts,mathrsfs,verbatim}
\usepackage[all]{xy}
\usepackage[mathcal]{euscript}
\usepackage{epsfig}
\usepackage{times}
\usepackage{hyperref}

\newcommand{\N}{{\mathbf N}}                   %Non-negative integers
\newcommand{\Z}{{\mathbf Z}}                   %Integers
\newcommand{\R}{{\mathbf R}}                   %Reals
\newcommand{\C}{{\mathbf C}}                   %Complex line
\newcommand{\CP}{\mathbf{P}^1}                %Riemann sphere 
\renewcommand{\H}{{\mathbf H}}                   %Hypercohomology
\newcommand{\Ct}{\widehat{\mathbf C}}          %Dual complex line
\newcommand{\CPt}{\widehat{\mathbf{P}}^1}     %Dual Riemann sphere 
\newcommand{\E}{E}                  		%Underlying holomorphic bundle --- DE RHAM CASE	!!!!
\newcommand{\F}{F}              	    %Modified holomorphic bundle	--- DE RHAM CASE	!!!!
\newcommand{\G}{G}		                  %Modified holomorphic bundle	--- DE RHAM CASE	!!!!
\newcommand{\Vt}{\widehat{V}}			%Dual smooth vector bundle
\newcommand{\Et}{\widehat{E}}		        %Dual holomorphic bundle	--- DE RHAM CASE	!!!!
\newcommand{\Ft}{\widehat{F}}		       	%Dual filtration
\newcommand{\Pt}{\widehat{P}}        		%Dual parabolic set
\newcommand{\Dt}{\widehat{D}}        		%Dual connection
\renewcommand{\tt}{\widehat{\theta}}                   %Dual Higgs field
\newcommand{\htr}{\widehat{h}}        		%Dual Hermitian metric
\newcommand{\nablat}{\widehat{\nabla}}		%Dual holomorphic connection
\renewcommand{\O}{{\mathcal O}}                 %Holomorphic functions
\newcommand{\Nahm}{{\mathcal N}}              %Nahm transform map
\renewcommand{\d}{\mbox{d}}                      %Differential
\newcommand{\gl}{\mathfrak{gl}}                %Lie algebra gl(r,C)

\newcommand{\h}{{\mathfrak h}}
\newcommand{\e}{e}			%Trivialisation of $\E$		--- DE RHAM CASE	!!!!
\newcommand{\f}{f}			%Trivialisation of $\F$		--- DE RHAM CASE	!!!!
\newcommand{\g}{g}			%Trivialisation of $\G$		--- DE RHAM CASE	!!!!
\newcommand{\dbar}{\bar{\partial}}

\DeclareMathOperator{\res}{res}
\DeclareMathOperator{\Gl}{Gl}
\DeclareMathOperator{\Gr}{Gr}
\DeclareMathOperator{\diag}{diag}
\DeclareMathOperator{\Id}{Id}
\DeclareMathOperator{\coker}{coker}
\DeclareMathOperator{\im}{im}

\DeclareMathOperator{\rank}{rk}

\DeclareMathOperator{\red}{red}

\DeclareMathOperator{\dR}{dR}
\DeclareMathOperator{\Dol}{Dol}
\DeclareMathOperator{\para}{par}

\newtheorem{prop}{Proposition}[section]
\newtheorem{assn}[prop]{Assumption}
\newtheorem{rk}[prop]{Remark}
\newtheorem{clm}[prop]{Claim}
\newtheorem{lem}[prop]{Lemma}
\newtheorem{defn}[prop]{Definition}
\newtheorem{cor}[prop]{Corollary}
\newtheorem{thm}[prop]{Theorem}

\newtheorem{example}{Example}

\title[Nahm transformation]
{Nahm transformation for parabolic integrable connections on the projective line --- case of generic regular graded residues}
\author{Szil\'ard Szab\'o}

\begin{document}

\maketitle

\begin{abstract}
We give a de Rham interpretation of Nahm's transform for certain parabolic harmonic bundles on the projective line 
and compare it to minimal Fourier--Laplace transform of $\mathcal{D}$-modules. We give an algebraic definition 
of a parabolic structure on the transformed bundle and show that it is compatible with the transformed harmonic metric. 
\end{abstract}

\section{Introduction}
In this paper we continue our study of Nahm transformation for some singular solutions of Hitchin's equations \cite{Hit} on the complex projective line 
with finitely many logarithmic singularities and one singularity of Poincar\'e rank (Katz-invariant) $1$. 
These solutions of Hitchin's equations were studied (in higher degree of generality) by O. Biquard and P. Boalch \cite{Biq-Boa}, 
generalizing the work of C. Simpson \cite{Sim}.

Let $\C \subset \CP$ denote the complex affine and projective lines, endowed with the Euclidean metric, 
and with standard holomorphic coordinate denoted by $z\in \C$. 
We consider a parabolic harmonic bundle $(V, F_i^j , \dbar^{\E}, D , h)$ on $\CP$ with logarithmic singularities at some fixed points 
$z_1,\ldots,z_n\in\C$ and a second-order pole with semi-simple leading order term at infinity. 
Let $\Ct$ be a different copy of the complex affine line with coordinate $\zeta$, and $\CPt$ the associated projective line. 
The aim of this paper is to construct a transformed %Higgs bundle $(\Vt, \dbar^{\Et} , \tt )$ 
harmonic bundle $(\Vt, \dbar^{\Et} , \tt , \htr)$ on $\CPt$ and study the properties of the mapping 
%\begin{equation}\label{eq:Nahm}
%    \Nahm: (V, F_i^j, \dbar^{\E},\theta ) \mapsto (\Vt, \Ft_i^j, \dbar^{\Et} ,\tt ),
%\end{equation}
\begin{equation}\label{eq:Nahm}
    \Nahm: (V, F_i^j, \dbar^{\E}, \nabla , h) \mapsto (\Vt, \Ft_i^j, \dbar^{\Et} ,\nablat , \htr),
\end{equation}
called Nahm transformation, on moduli spaces. 

In the case where the residues of $D$ at the singular points are semisimple, the transform was defined in \cite{Sz-these}, 
and its properties were further studied in \cite{Sz-laplace}, \cite{A-Sz}, \cite{Sz-plancherel}. 
Part of the construction was extended to the general case in \cite{Sz-FM34}. 
Namely, in \cite{Sz-FM34} we provided a construction of the parabolic Higgs bundle underlying the transformed solution. 
For this purpose, we extended the Fredholm-theory and the Dolbeault hypercohomology interpretation of the first 
$L^2$-cohomology of a twisted elliptic complex. 

In this paper, after giving some background material on parabolic harmonic bundles in Section \ref{sec:harmonic-bundles} and on 
Nahm transform for parabolic Higgs bundles in Section \ref{sec:Nahm}, we define in Section \ref{sec:transformed-connection} the 
transformed flat connection and harmonic metric. 
Then, in Section \ref{sec:dR} we give a de Rham hypercohomology interpretation of the twisted elliptic complex, 
which then leads the way to an algebraic construction of the transformed parabolic structure carried out in Section \ref{sec:transformed-parabolic}. 
Section \ref{sec:Laplace} (and in particular, Theorem \ref{thm:main}) provides an identification between Nahm transformation and 
Fourier--Laplace transformation of $\mathcal{D}$-modules studied for instance in \cite{Mal} and \cite{Sab-isomonodromy}, 
and in Corollary \ref{cor:dR-isom} we prove that Nahm transformation is a holomorphic isomorphism between de Rham moduli spaces. 
The generic transformation of the singularity parameters under a regularity and a suitable genericity assumption on the eigenvalues 
of the residues is given in Section \ref{sec:transformation}. 
Finally, in Section \ref{sec:parabolic} we show that the transformed harmonic metric $\htr$ is compatible with the parabolic structure 
in the sense that local sections of the various pieces of the filtered vector bundles have well controlled local behaviour near the 
parabolic points. 

The topic of this paper is closely related to recent work \cite{DPS} by R. Donagi, T. Pantev and C. Simpson on push-forward of 
logarithmic parabolic Higgs bundles under a map from a projective surface endowed with a divisor of some special type to a curve, 
using C. Sabbah's work \cite{Sab-twistor} on twistor $\mathcal{D}$-modules and its extension \cite{Moc} by T. Mochizuki's to the parabolic case. 
Indeed, Nahm transformation can be roughly defined in the spirit of other integral transforms, as a composition of 
the following functors: pull-back to a product surface, tensor by a rank $1$ solution, and push-forward with respect to the second projection map. 
It turns out that transforming the compatible parabolic structure requires most of the ideas here, in particular one needs a direct image functor 
for parabolic structures, which is precisely the content of \cite{DPS}. 
In the paper \cite{DPS} the authors assume a nilpotence hypothesis for the residue of the Higgs field. 
Curiously, this appears to be more or less complementary to our assumptions (see Proposition \ref{prop:par_str} and Remark \ref{rem:Assumptions}). 
It seems a natural guess that combining our techniques with those of \cite{DPS} it would be possible to lift at least some of our assumptions.  
We plan on returning to such a generalization of this paper in a joint work with T. Mochizuki.

\section{Parabolic harmonic bundles}\label{sec:harmonic-bundles}

In this section we recall generalities about some singular solutions of Hitchin's equations on a curve. 

Let $C$ be a complex analytic curve. We denote by $\O_C$ and $K_C$ its structure sheaf and its canonical sheaf respectively, 
and by $\Omega^k$ the sheaf of locally $L^2$ differential $k$-forms on $C$. 

Let $V$ be a smooth vector bundle over $C$ of rank $r\geq 2$ and $\E$ be a holomorphic vector bundle with underlying smooth vector bundle $V$. 
The space of local sections of $\E$ may be conveniently described as the kernel of a partial differential operator $\dbar^{\E}$ 
of type $(0,1)$ on $V$. Let 
$$
  \nabla : \E \to \E \otimes_{\O_C} K_C
$$
be a (possibly meromorphic) connection on $\E$. Namely, we assume that $\nabla$ satisfies the Leibniz rule: 
for any open set $U \subset C$ and $f\in \O (U), \e \in \E(U)$ we have 
$$
  \nabla (f \e ) = (\d f) \e + f \nabla (\e ),
$$
where $\d$ denotes exterior differential. 
The couple $(\E, \nabla)$ is then called an integrable bundle. 
We also set 
$$
  D = \dbar^{\E} + \nabla, 
$$
which is a flat connection on $V$. The data $(\E, \nabla)$ determines and is uniquely determined by $(V, D)$. 
Let $h$ be a smooth fibrewise Hermitian metric on $V$. 
Split $D$ as 
$$
  D = D^+ + \Phi
$$
with $D^+$ unitary and $\Phi$ self-adjoint with respect to $h$, and split these operators further with respect to bidegree: 
\begin{align*}
  D^+ & = \partial^h + \bar{\partial}^{\mathcal E}  \\
  \Phi & = \theta + \theta^*, 
\end{align*}
with $\partial^h, \theta$ of type $(1,0)$ and $\bar{\partial}^{\mathcal E}, \theta^*$ of type $(0,1)$, and where $\theta^*$ is the 
adjoint of $\theta$ with respect to $h$. 
With these notations, $(V, D , h)$ is called a harmonic bundle if and only if 
\begin{equation}\label{eq:theta-holomorphic}
  \dbar^{\mathcal{E}} \theta = 0, 
\end{equation}
i.e. $\theta$ is a holomorphic (or meromorphic) $K$-valued endomorphism. 

From now on, we let $(V, D , h)$ (or equivalently, $(E, \nabla, h)$) denote a harmonic bundle over $\CP$ with some singularities. 
We will now spell out explicitly our assumptions on its singularities, 
as well as the definition of a compatible parabolic structure $F_i^j$. 
Fix distinct points $z_1,\ldots,z_n\in\C$, and denote by  $z_0= [0:1] \in \CP$ the point at infinity. 
Let $\E$ be given the structure of a quasi-parabolic bundle on $\C$ with parabolic points 
$z_0,z_1,\ldots,z_n$, i.e. for every $i\in\{0,\ldots,n\}$ a decreasing filtration of $\C$-vector 
subspaces of the fiber of $V$ at $z_i$
\begin{equation}\label{eq:parfiltr}
	\{ 0\} = F_i^{l_i} \subset F_i^{l_i-1} \subset \cdots \subset F_i^1 \subset F_i^0 = V_{z_i}
\end{equation}
of some length $1\leq l_i\leq r$ is given. 
For $i\in\{0,\ldots,n\},j\in\{ 0,\ldots,l_i-1\}$ denote the graded vector spaces associated to (\ref{eq:parfiltr}) by 
\begin{equation}\label{eq:pargr}
	\Gr_{i}^j = \Gr_{F_i}^j = F_i^j/F_i^{j+1}.
\end{equation}
We fix parabolic weights $\{ \beta_i^j \}$ for $i\in\{0,\ldots,n\}, j\in\{ 0,\ldots,l_i-1\}$ satisfying 
\begin{equation}\label{eq:parweights}
	1> \beta_i^{l_i-1} > \cdots > \beta_i^0 \geq 0. 
\end{equation}
For every $0 \leq i \leq n$ we will take a local holomorphic trivialisation $\{ \e_i^s\}_{s=1}^r$ of $\E$ near $z_i$ 
compatible with the filtration $F_i^j$ in the sense that $F_i^j$ is spanned by the evaluations at $z_i$ of the vectors 
$$
  \e_i^1, \ldots, \e_i^{\dim F_i^j}. 
$$
With respect to such a compatible basis, we will use the diagonal matrix 
$$
  \diag(\beta_i^j )_{j=0}^{l_i-1}
$$
consisting of the parabolic weights, each $\beta_i^j$ repeated with multiplicity equal to $\dim \Gr_{i}^j$. 
We assume that $\nabla$ is a logarithmic connection in $E$ over $\C$ with logarithmic singularities at $z_1, \ldots ,z_n$ 
and an irregular singularity of Katz-invariant $1$ at infinity, compatible with the parabolic structure. 
By compatibility in the logarithmic case we mean that the residue 
\begin{equation}\label{eq:residue}
	\res_{z_i}(\nabla ) = \nabla %\righthalfcup 
	((z-z_i)\partial_z)
\end{equation}
of $\nabla$ at $z_i$ preserves the filtration $F_i^{\bullet}$: 
\begin{equation}\label{eq:logHiggsfield}
	\res_{z_i}( \nabla ) : F_i^j \to F_i^j
\end{equation}
for every $i\in\{1,\ldots,n\}, j\in\{ 0,\ldots,l_i-1\}$ . 
At the singularity $z_0$, in a suitable local holomorphic trivialisation of $E$ we require that $\nabla$ is given by a convergent Laurent series
\begin{equation}\label{eq:Dinfinity}
 \nabla = \d^{1,0} + A \d z + C \frac{\d z}z + O(z^{-2}) \d z, 
\end{equation}
where $A, C \in \gl_r (\C )$ are matrices of dimension $r$ preserving the filtration $F_0^{\bullet}$, 
and moreover such that $A$ is semi-simple. 
Without loss of generality, we may then assume that $A$ is diagonal, and we denote by $\Pt$ the set of eigenvalues of $A$. 
Let us denote by $H$ the centraliser of $A$ in $\Gl_r(\C)$  and by $\h$ its Lie-algebra. 
Then up to applying a holomorphic gauge transformation near $\infty$ we can arrange that $C\in\h$; 
in what follows we will therefore assume $C\in\h$. 

By compatibility, $\res_{z_i}(\nabla )$ acts on the spaces (\ref{eq:pargr}). Let us denote by $\res_{z_i}(\nabla )^j$ this action and let 
$$
	\res_{z_i}(\nabla )^j = S_i^j  + N_i^j 
$$
be its decomposition into its semi-simple and nilpotent components respectively. 
We may (and henceforth will) assume that the compatible trivialisations $\{ \e_i^s\}_{s=1}^r$ are chosen so that $S_i^j$ are diagonal for each $i,j$. 
The generalized eigenspaces of $S_i^j$ then define a block-decomposition of $\Gr_i^j$. 
To each such block there corresponds a single eigenvalue of $\res_{z_i}(\nabla )^j$, and the eigenvalues are different on different blocks. 

Now, there exists a weight filtration $W_{i,\bullet}^j$ of $\Gr_{F_i}^j$ associated to $N_i^j$; 
for a list of its (well-known) properties, see \cite{Sz-FM34}. 
We will work with local trivializations $\{ \e_i^s\}_{s=1}^r$ of $\E$ near $z_i$ that preserve these filtrations too. 
Namely, for any vector $\e_i^s \in F_i^j \setminus F_i^{j+1}$ we require 
$$
  N_i^j (\e_i^s (z_i ) ) = \e_i^{s+1} (z_i ) 
$$
unless $\e_i^s(z_i )\in \ker (N_i^j )$. 

The connection form of $\nabla$ with respect to the local analytic trivialization $\e_i^s$ near $z_i$ is of the form 
\begin{equation}\label{eq:Aiz}
  \frac{A_i(z)}{z-z_i} \d z 
\end{equation}
for some $\gl_r(\C)$-valued analytic function $A_i$. 
In addition, up to applying an analytic gauge transformation, we may assume that the off-diagonal entries $a_{i,s's}$ of $A_i$ 
identically vanish for all $s,s'$ corresponding to eigenvalues $\mu_i^s,\mu_i^{s'}$ of $A_i(0)$ satisfying $\mu_i^s - \mu_i^{s'}\notin \Z$. 
This can be proved along the lines of Propositions 2.11 and 2.13 \cite{Sab-isomonodromy}: indeed, the equation for the corresponding entry of the 
$l$'th coefficient $P_l$ of the local analytic gauge transformation to solve is 
$$
  (\mu_i^s - \mu_i^{s'} + l ) P_{l,s's} = * 
$$
for some polynomial expression of the entries of the previous coefficients $P_1,\ldots ,P_{l-1}$ and of $P_{l,r'r}$ for $|r'-r| < |s'-s|$ on the right-hand side; 
this recursion is solvable for any $l$ by the assumption $\mu_i^s - \mu_i^{s'}\notin \Z$ and the solution is convergent for the same reason as the one found in [{\it op. cit.}]. 
An important observation however is that $A_i$ may have strictly upper diagonal block entries with respect to the filtration $F_i$: indeed, for 
$\e_i^s \in F_i^j, \e_i^{s'} \in F_i^{j'}$ with $j' > j$ such that the corresponding eigenvalues $\mu_i^s,\mu_i^{s'}$ 
differ by an integer, the entry $a_{i,s's}$ may be non-zero. 

We assume given a Hermitian metric $h$ in $\E$ in some neighborhood of $z_i$ ($i\in\{1,\ldots,n\}$) compatible with $(E,\nabla )$; 
for the definition of compatibility, and the relationship between the singularity parameters of $\nabla$ and $\theta$ see \cite{Sim} 
(or \cite{Sz-FM34}). 
From now on we let $(V, F_i^j, \dbar^{\E}, \nabla , h)$ denote a harmonic bundle on $\CP$ with parabolic structure and admissible harmonic metric, 
and singularity behaviour fixed as above. We now make important assumptions necessary to carry out our construction. These assumptions are 
less stringent as the ones appearing in Definitions 1.1 and 1.2 \cite{Sz-laplace}. 
\begin{assn}\label{assn:main}
 For any $i,j$
 \begin{enumerate}
  \item \label{assn:main0}
  if $i=0$ then no eigenvalue $\mu_i^s$ of $\res_{z_i}(\nabla )^j$ is an integer;
  \item \label{assn:main1}
   for $i > 0$ if $\beta_i^0=0$ then $S_i^0$ has no non-zero integer eigenvalue and the nilpotent part $N_i^0$ of the endomorphism $\res_{z_i}(\nabla )^0$ 
   acts trivially on the $0$-eigenspace of $\res_{z_i}(\nabla )^0$; 
 \item \label{assn:main2} 
 for $i > 0$ if $\beta_i^j > 0$ then $\res_{z_i}(\nabla )^j$ has no integer eigenvalue.
  \end{enumerate}
\end{assn}

In particular, in view of the previous paragraph and part (\ref{assn:main2}), for $i > 0$ any entry $a_{i,s's}$ of the connection matrix $A_i(z)$ 
corresponding to a pair of indices $s,s'$ such that $\beta_i^{j(s')} = 0 = \mu_i^{s'}$ and $\beta_i^{j(s')} > 0$ (or vice versa) identically vanishes. 

In what follows we will call singular (or parabolic) harmonic bundle a tuple $(V, F_i^j, \dbar^{\E}, \nabla , h)$ 
satisfying the equations and local behaviours fixed in this section.

\section{Construction of Nahm transformation}\label{sec:Nahm}
In this section we briefly explain the part of the construction of the Nahm transform of a parabolic harmonic bundle 
$(V, F_i^j, \dbar^{\E}, D , h)$ over $\CP$ %, with singularities as fixed in Section \ref{sec:harmonic-bundles}, 
carried out in \cite{Sz-FM34}. 

With the above notations, given any $\zeta \in \Ct \setminus \Pt$ we define a twisted flat connection on $\CP$ by 
\begin{equation}\label{Dzeta}
  D_{\zeta} = D - \zeta \d z,
\end{equation}
and consider the twisted $L^2$ elliptic complex 
\begin{equation}\label{eq:elliptic-complex-Dzeta}
 	0\to V\xrightarrow{D_{\zeta}}  V\otimes \Omega^1 \xrightarrow{D_{\zeta}}  V\otimes \Omega^2 \to 0, 
\end{equation}
where the $L^2$ conditions are defined using the Euclidean metric on $\C$ and the fiber metric $h$. 
The cohomology spaces of this complex of degrees $0$ and $2$ vanish for all $\zeta \in \Ct \setminus \Pt$, 
and its cohomology spaces of degree $1$ are of constant finite dimension. 
Hodge theory then provides an equivalent characterization of $\Vt_{\zeta}$ as the space of $V$-valued harmonic $1$-forms, 
i.e. $1$-forms annihilated by the twisted Dirac--Laplace operator 
\begin{equation}\label{eq:Dirac-Laplace}
   \Delta_{\zeta} = D_{\zeta} D_{\zeta}^* + D_{\zeta}^* D_{\zeta}. 
\end{equation}
It follows that the first $L^2$-cohomology spaces of (\ref{eq:elliptic-complex-Dzeta}) form a smooth vector bundle 
$\Vt\to \Ct \setminus \Pt$. The vector bundle $\Vt$ is the smooth vector bundle underlying the Nahm transform of 
$(V, F_i^j, \dbar^{\E}, D , h)$. 
A Dolbeault resolution of suitable locally free sheaves of $\O_{\CP}$-modules $\mathcal{F}, \mathcal{G}$ 
yields yet another description of $\Vt_{\zeta}$ as the first hypercohomology space of a complex 
$$
	\mathcal{F} \xrightarrow{\theta} \mathcal{G} \otimes K_{\CP}(2\cdot z_0 + z_1 + \cdots + z_n). 
$$
This identification equips $\Vt$ with the structure of a holomorphic vector bundle $\widehat{\mathcal E}$ 
and with a natural extension of $\Vt$ to $\CPt$. 
Using the parabolic filtrations $F_i^j$, we then refine the definitions of $\mathcal{F}$ and $\mathcal{G}$ in the complex 
above in order to define a transformed parabolic structure $\Ft_i^j$. 
Moreover, multiplication by $-z/2 \d \zeta$ induces an endomorphism-valued $1$-form $\widehat{\theta}$ of $\widehat{\mathcal E}$. 
According to Theorem 6.2 of \cite{Sz-FM34}, the outcome of these constructions is a parabolic Higgs bundle 
\begin{equation}\label{eq:transformed-Higgs}
   (\widehat{\mathcal E}, \widehat{\theta}, \Ft_i^j )
\end{equation}
over $\CPt$, with $\widehat{\theta}$ having a first-order pole at the points of $\Pt$ and a second-order pole at 
$\infty \in \CPt$ with semi-simple leading-order term, endowed with a compatible parabolic structure $\Ft_i^j$ 
at the points $\Pt \cup \{ \infty \}$.

\section{Transformed flat connection and Hermitian metric}\label{sec:transformed-connection}

In this section we will construct a transformed flat connection $\Dt$ and Hermitian metric $\htr$ on $\Vt$ over $\Ct \setminus \Pt$, 
which completes (\ref{eq:transformed-Higgs}) into an irregular harmonic bundle, as in \cite{Biq-Boa}. 
These constructions agree with the ones given in Section 3.1 of \cite{Sz-these}. 

Let us start by defining $\htr$. Let  $\zeta \in \Ct \setminus \Pt$ be arbitrary and consider $\hat{f}_1, \hat{f}_2 \in \Vt_{\zeta}$. 
As explained in Section \ref{sec:Nahm}, an element $\hat{f} \in \Vt_{\zeta}$ can be uniquely represented by a $1$-form 
$\hat{f}(z) \d z + \hat{g}(z) \d \bar{z}$ with values in $V$ in the kernel of the operator (\ref{eq:Dirac-Laplace}) and such that 
$h(\hat{f}(z),\hat{f}(z))$ is integrable over $\Ct$ 
with respect to the Euclidean norm. For two such elements it is natural to set 
\begin{equation}\label{eq:transformed_metric}
   \htr (\hat{f}_1, \hat{f}_2) = \int_{\C}   h(\hat{f}_1(z),\hat{f}_2(z)) + h(\hat{g}_1(z),\hat{g}_2(z)) | \d z |^2 
\end{equation}
This formula defines a Hermitian metric on $\Vt_{\zeta}$, and as these subspaces vary smoothly within the space $L^2(\C, V\otimes \Omega^1 )$ of all 
square-integrable $V$-valued $1$-forms, it follows that the formula above yields a smooth fiber metric over the bundle $\Vt \to \Ct \setminus \Pt$. 

Let us now come to the construction of $\Dt$. For this purpose, introduce the orthogonal projection operator 
$$
  \hat{\pi}_{\zeta} : L^2(\C, V\otimes \Omega^1 ) \to \Vt_{\zeta}, 
$$
and let $\widehat{\d}$ denote the trivial connection on the trivial Hilbert-space bundle 
\begin{equation}\label{eq:Hilbert-bundle}
  L^2(\C, V\otimes \Omega^1 ) \times \Ct \setminus \Pt \to \Ct \setminus \Pt. 
\end{equation}
With this notation, we set 
\begin{equation}\label{eq:Dt}
 \Dt = \hat{\pi}_{\zeta}  \circ (\widehat{\d} - z \d \zeta) 
\end{equation}
where the operator $z$ acts on $\hat{f}$ by multiplying it by $z$. 

\begin{thm}
 The connection $\Dt$ defines a flat connection on $\Vt$ and $\htr$ is a harmonic metric for $\Dt$. 
\end{thm}

\begin{proof}
 See Proposition 3.5 and Theorem 4.9 \cite{Sz-these}. 
\end{proof}

\section{de Rham interpretation}\label{sec:dR}

Parallel to the Dolbeault interpretation of Section 4 \cite{Sz-FM34}, there is also a de Rham hypercohomology 
interpretation of the transform in terms of a complex of $\C$-vector spaces 
\begin{equation}\label{eq:dRcopml}
	\F \xrightarrow{\nabla_{\zeta}} \G \otimes K_{\CP}(2\cdot z_0 + z_1 + \cdots + z_n). 
\end{equation}
In this section we will write down this interpretation. 

For this purpose, we merely need to suitably define the sub-sheaves $F,G \subset E$ 
so that (\ref{eq:dRcopml}) admit an $L^2$ resolution for the Euclidean metric. 
The sub-sheaves $F,G$ will be defined as elementary transforms of $E$ along some sub-spaces of the fibers of $E$ at 
the points $z_0, z_1,\ldots,z_n$. In particular, they are defined so that they agree with $E$ away from these points. 
We now come to writing down local frames $\{ \f_i^s\}_{s=1}^r$ of $F$ and $\{ \g_i^s\}_{s=1}^r$ of $G$ near the parabolic points 
in terms of local frames $\{ \e_i^s\}_{s=1}^r$ compatible with the filtrations $F_i^{\bullet}$ and $W_{i,\bullet}^j$ 
and with respect to which the semi-simple part $S_i^j$ of $\res_{z_i}{\nabla }^j$ is diagonal, as in Section \ref{sec:harmonic-bundles}. 
We introduce the following notations: for any $1\leq s \leq r$ we let $j_i(s)$ stand for the only $0\leq j \leq l_i - 1$ satisfying 
$$
  \e_i^s \in F_i^j \setminus F_i^{j+1}, 
$$
and $k_i(s)$ for the only integer $k$ such that 
$$
  \e_i^s \in W_{i,k}^j \setminus W_{i,k-1}^j. 
$$
For ease of notation we will drop the subscript $i$ of $j_i(s)$ and $k_i(s)$; hopefully, this will cause no misunderstanding. 
We begin by spelling out a frame of $G$ in the case $i>0$; the formulas are analogous to the ones of Section 4 \cite{Sz-FM34} in 
the case of the Dolbeault interpretation: 
\begin{enumerate}
\item Case $\beta_i^{j(s)}=0$
\begin{enumerate}
	\item sub-case $k(s) < -1$: set $\g_i^s=\e_i^s$
	\item sub-case $k(s)\geq -1$: set $\g_i^s=(z-z_i)\e_i^s$
\end{enumerate}
\item Case $\beta_i^{j(s)}> 0$: set $\g_i^s=\e_i^s$. 
\end{enumerate}
We now come to a local frame of $G$ near $z_0$: 
\begin{enumerate}
\item Case $\beta_0^{j(s)}=0$
\begin{enumerate}
	\item sub-case $k(s) < -1$: set $\g_0^s=z^{-1}\e_0^s$
	\item sub-case $k(s)\geq -1$: set $\g_0^s=z^{-2}\e_0^s$
\end{enumerate}
\item Case $\beta_i^{j(s)}> 0$: set $\g_0^s=z^{-1}\e_0^s$. 
\end{enumerate}
Next, let us denote by $\mu_i^s$ the eigenvalue of $\res_i(\nabla)^{j(s)}$ corresponding to the eigenvector $\e_i^s$. 
With this notation, we define in the case $i>0$: 
\begin{enumerate}
\item Case $\beta_i^{j(s)}=0$
\begin{enumerate}
%	\item sub-case $\lambda_i^{s}=0, N_{i,j(s)}(e_i^s(z_i))=0, k(s)<-1$: set $\f_i^s = (z-z_i)^{-1}\e_i^s$
	\item sub-case $\mu_i^{s}=0$ (and as a consequence necessarily $\res_{z_i}^{j(s)}(e_i^s(z_i))=0%, k(s)\geq -1
											$ by Assumption \ref{assn:main} (\ref{assn:main1})): set $\f_i^s = \e_i^s$
%	\item sub-case $\lambda_i^{s}=0, N_i^{j(s)}(e_i^s(z_i))\neq 0, k(s)<1$: set $\f_i^s = \e_i^s$
%	\item sub-case $\lambda_i^{s}=0, N_i^{j(s)}(e_i^s(z_i))\neq 0, k(s)\geq 1$: set $\f_i^s = (z-z_i)\e_i^s$
	\item sub-case $\mu_i^{s}\neq 0, k(s)<-1$: set $\f_i^s = \e_i^s$
	\item sub-case $\mu_i^{s}\neq 0, k(s)\geq -1$: set $\f_i^s = (z-z_i)\e_i^s$
\end{enumerate}
\item Case $\beta_i^{j(s)}>0$: set $\f_i^s = \e_i^s$. 
%\begin{enumerate}
%	\item sub-case $\lambda_i^{s}=0, N_i^{j(s)}(e_i^s(z_i))=0$: set $\f_i^s = %(z-z_i)^{-1} %%%%%%% NEM LEHET: HA A FILTRALASBAN BAL-ALSO RESZEBEN THETANAK NEM 0 AZ $O(1)$ ELEME (AMI GENERIKUSAN IGAZ), AKKOR A KEPE NEM $L^2$
  %\e_i^s$
%	\item sub-case $\lambda_i^{s}=0, N_i^{j(s)}(e_i^s(z_i))\neq 0$: set $\f_i^s = \e_i^s$
%	\item sub-case $\lambda_i^{s}\neq 0$: set $\f_i^s = \e_i^s$. 
%\end{enumerate}
\end{enumerate}
Finally, we define $\F$ near $z_0$ by the local frame $\f_0^s = \g_0^s$ that we have already defined above. 

An important observation is that according to their definitions, $\F$ is naturally a subsheaf of $\G(z_1+\cdots + z_n)$. 
This will play a role in the next statement, as it is possible to add a multiple of the identity map of $\F$ to $\nabla$. 
Namely, we set 
\begin{equation}\label{eq:twisteddRmap}
   \nabla_{\zeta} = \nabla - \zeta \Id_{\E}\d z,
\end{equation}
where $\Id_{\E}$ stands for the identity map of $\E$ (and its restriction to the sub-sheaves of $\E$). 

We then have the following analogue of Proposition 4.1 \cite{Sz-FM34}. 
\begin{prop}\label{prop:L2dR}
For every $\zeta\in \Ct \setminus \Pt$, the fiber $\Vt_{\zeta}$ is isomorphic to the first hypercohomology space of the 
following twisted de Rham complex 
\begin{equation}\label{eq:twisteddRcopml}
	\F \xrightarrow{\nabla_{\zeta}} \G\otimes K_{\CP}(2\cdot z_0 + z_1+\cdots + z_n).
\end{equation}
\end{prop}

\begin{proof}
A straightforward computation shows that the twisted de Rham complex admits a resolution by sheaves of $L^2$ 
sections (in the domain of $\nabla_{\zeta}$ whenever this condition is applicable). 
As these sheaves are acyclic, the proof follows. 
\end{proof}

Proposition \ref{prop:L2dR} allows us to endow $\Vt$ with the structure of a holomorphic vector bundle $\Et$ over $\Ct \setminus \Pt$. 
Namely, using $\pi_j$ for the projection morphism to the $j$'th factor in $\CP \times \CPt$ (and on its various open subsets), we may set 
\begin{equation}\label{eq:extension}
   \Et = \mathbf{R}^1 (\pi_2)_* \left( \pi_1^* \F \xrightarrow{\nabla_{\zeta}} \pi_1^*\G\otimes K_{\CP}(2\cdot z_0 + z_1+\cdots + z_n) \right). 
\end{equation}
Then, as $\nabla_{\zeta}$ depends analytically on $\zeta$, $\Et$ naturally inherits the structure of a coherent analytic sheaf 
of $\O_{\Ct\setminus \Pt}$-modules. Now, as the dimensions of the fibers $\Et_{\zeta}$ of $\Et$ is independent of $\zeta$, it follows that 
$\Et$ is the sheaf of sections of an analytic vector bundle over $\Ct \setminus \Pt$. 
Moreover, comparing this to the definition (\ref{eq:Dt}) of $\Dt$, one immediately sees that the two constructions 
are compatible in the sense that the $\dbar$-operator $\Dt^{(0,1)}$ annihilates precisely the sections of the holomorphic bundle $\Et$. 
Indeed, the $(0,1)$-part of $\Dt$ is induced by the trivial $\dbar$-operator in the bundle of Hilbert spaces (\ref{eq:Hilbert-bundle}) 
with respect to the variable $\zeta$, and the holomorphic structure of $\Et$ is also induced by the trivial holomorphic structure of 
$\pi_1^*\F, \pi_1^*\G$ with respect to $\zeta$.

\section{Extension of the transformed vector bundle and transformed parabolic structure}\label{sec:transformed-parabolic}

The description of $\Vt_{\zeta}$ given in Proposition \ref{prop:L2dR} in terms of a twisted de Rham complex allows us to 
extend the smooth vector bundle $\Vt$ over the points $\Pt \cup \infty$ where its $L^2$-theoretic definition fails to exist. 
To be precise, we can extend the holomorphic bundle $\Et$. 
Moreover, using the same kind of ideas it is possible to endow the extension with a parabolic structure. 
The purpose of this section is to spell out this extension and parabolic structure.  
The content of this section closely follows Section 5 of \cite{Sz-FM34}. 

\subsection{Extension of the transformed vector bundle over $\CPt$}
We define the extension of $\Et$ to $\Ct$ simply by the formula (\ref{eq:extension}). As $\nabla_{\zeta}$ depends analytically (in fact, even algebraically) 
on $\zeta$, this obviously defines a coherent analytic sheaf on $\Ct$. 

Let us now define the extension at $\infty$. We consider the global sections $s_0, s_{\infty} \in H^0(\CPt, \O_{\CPt}(1))$ such that on the affine chart $\Ct \subset \CPt$ we have 
$$
  s_0 (\zeta ) = \zeta, \quad s_{\infty} (\zeta ) = 1. 
$$
Let $\pi_i$ stand for the projection from $ \CP \times \CPt$ to its $i$'th factor and $\Id_{\E}$ the sheaf morphism induced by the identity of $\E$. 
We now modify (\ref{eq:twisteddRcopml}) into 
\begin{equation}\label{eq:nabla-extension}
   \nabla_{\zeta} = \nabla \otimes s_{\infty} - \Id_{\E}\d z \otimes s_0 : \pi_1^*\F \to \pi_1^*\G\otimes K_{\CP}(2\cdot z_0 + z_1+\cdots + z_n) \otimes \pi_2^*  \O_{\CPt}(1). 
\end{equation}
This is then clearly a holomorphic deformation of $\nabla$ parametrized by $\CPt$, in particular its index is constant over $\CPt$. 
\begin{prop}\label{prop:hypercohomology}
 The hypercohomology groups of degree $0$ and $2$ of (\ref{eq:nabla-extension}) vanish for all $\zeta \in \CPt$. 
\end{prop}

\begin{proof}
 A non-trivial degree $0$ hypercohomology class is represented by a global holomorphic section $e\in \Gamma (\CP , F)$ 
 that is parallel with respect to $\nabla_{\zeta}$. Let us first treat the case of $\zeta \in \Ct$. Then, $\nabla$ preserves 
 the subsheaf $L \subset F$ generated by $e$. Namely, we have $\nabla (e) = \zeta e \d z$. This implies that $e$ is an eigenvector 
 both for the leading-order term $A$ and for the residue $C$ of $\nabla$ at infinity, see (\ref{eq:Dinfinity}). 
 In addition, the corresponding eigenvalue of $C$ is clearly $0$. 
 On the other hand, it is well-known that the residue must be of the form $\mu_i^s + k$ for some $1\leq s \leq r$ 
 and some integer $k\in \Z$. This contradicts Assumption \ref{assn:main} (\ref{assn:main0}), hence the hypercohomology 
 group of degree $0$ is trivial for $\zeta \in \Ct$. 
 For $\zeta = \infty$, (\ref{eq:nabla-extension}) reduces to $-\Id_{\E}\d z$, and a degree $0$ hypercohomology class is 
 represented by a global section $e\in \Gamma (\CP , F)$ annihilated by $\Id_{\E}$. As the stalks at $\infty$ of the sheaves 
 $F,G$ agree, the only such section is $e=0$ near $\infty$. Then, by analytic contnuation, $e=0$ on $\CP$.   
 
 The case of degree $2$ can be treated similarly. 
\end{proof}

We introduce the symbol $\pi_2$ for the projection onto the second factor 
$$
   \pi_2: \CP \times \CPt \to \CPt.
$$
Let us now denote by $\dR_{\bullet}$ the family (\ref{eq:nabla-extension}) of complexes of $\C$-vector spaces over $\CP$ 
parametrized by $\zeta \in \CPt$ and set 
\begin{equation}\label{eq:transformed_bundle}
   \Et = \R^1 (\pi_2)_* \dR_{\bullet} . 
\end{equation}

\begin{prop}\label{prop:locally_free}
 The sheaf of $\O_{\CPt}$-modules $\Et$ is locally free, with fiber over $\zeta \in \CPt$ given by $\H^1 (\dR_{\zeta})$. 
\end{prop}
\begin{proof}
In view of the proposition and because the index of a continuous deformation of a sheaf morphism is locally constant, 
the dimension of the first hypercohomology spaces $\H^1(\nabla_{\zeta})$ of the twisted de Rham complex (\ref{eq:nabla-extension}) 
is independent of $\zeta\in\CPt$. 
The family $\dR_{\bullet}$ of complexes is clearly flat over $\CPt$. 
%On the other hand, these vector spaces clearly form a coherent analytic sheaf over $\CPt$. 
The statements then follow from the proper base change theorem. % is in fact an analytic vector bundle. 
\end{proof}

\begin{defn}
 The extension of $\Et$ over $\CPt$ is the holomorphic vector bundle whose fiber over $\zeta \in \CPt$ is defined as the first 
 hypercohomology space $\H^1(\nabla_{\zeta})$ of the twisted de Rham complex (\ref{eq:nabla-extension}). 
\end{defn}
By an abuse of notation, we will continue to denote the extension of $\Et$ over $\CPt$ by $\Et$.

\subsection{Transforming the parabolic structure}\label{ssec:transforming-parabolic}

We define the transformed parabolic structure by first refining the definition of $\F$ and $\G$ from Section \ref{sec:dR} to depend on a parameter $\beta\in \R$. 
Namely, for any $1 \leq i \leq n$ 
\begin{enumerate}
 \item \label{case:deletion}
 for the values $1\leq s \leq r$ such that $\beta_i^{j(s)} = 0 = \mu_i^s$ and \emph{any} $\beta\in\R$ we set 
  $$
    \f_i^s (\beta ) = \f_i^s, \quad \g_i^s (\beta ) = \g_i^s;
 $$
 \item \label{case:beta-shifting} 
 for the values $1\leq s \leq r$ such that at least one of $\beta_i^{j(s)}, \mu_i^s$ is non-zero and for the unique integer $m$ satisfying 
 $\beta_i^{j(s)} + m -1 < \beta \leq \beta_i^{j(s)} + m$ we set %\label{case:beta-smaller}
 $$
    \f_i^s (\beta ) = (z-z_i)^m \f_i^s, \quad \g_i^s (\beta ) = (z-z_i)^m\g_i^s.
 $$
% \item for the values $1\leq s \leq r$ such that at least one of $\beta_i^{j(s)}, \mu_i^s$ is non-zero and $\beta_i^{j(s)} < \beta$ we set \label{case:beta-greater}
% $$
%    \f_i^s (\beta ) = (z-z_i)\f_i^s, \quad \g_i^s (\beta ) = (z-z_i)\g_i^s. 
% $$
\end{enumerate}
In the case $i=0$ the definitions are similar to case (\ref{case:beta-shifting}), up to replacing the local coordinate $z-z_i$ by $z^{-1}$. 
We then let $\F_{\beta}, \G_{\beta}$ be the $\O_{\CP}$-modules that are equal to $\F, \G$ away from the points $z_i$, and that are generated by 
$\f_i^s (\beta ), \g_i^s (\beta )$ respectively near $z_i$ for all indices $1 \leq s \leq r$. 

\begin{prop}
 With the above definitions $\nabla$ induces for all $\beta \in \R$ a sheaf morphism 
$$
 \nabla_{\beta} : \F_{\beta} \to \G_{\beta} \otimes K_{\CP}(z_1+\cdots +z_n + 2\cdot z_0).
$$
\end{prop}
\begin{proof}
We only treat the case $i>0$, the case $i=0$ being similar and simpler. 
Using the notations introduced above, the action of $\nabla$ on $(z-z_i)^m \e_i^s$ with respect to the chosen trivialization $\e_i^{s'}$ is given by 
$$   
  \nabla ((z-z_i)^m \e_i^s) = (z-z_i)^{m-1} \left( \sum_{s' = 1}^r (a_{i,s's}(z) + m \delta_{s's}) \e_i^{s'}\right) \d z
$$
with $a_{i,s's}$ the entries of the regular matrix $A$ in (\ref{eq:Aiz}) and $\delta_{s's}$ standing for Kronecker's $\delta$-symbol. 
As we observed, $a_{i,s's} = 0$ unless $\mu_i^s - \mu_i^{s'} \in \Z$. 
Moreover since the residue is compatible with the parabolic structure, we also have $a_{i,s's}(z_i) = 0$ if $j(s) > j(s')$. 
Let us now write 
\begin{equation}\label{eq:nabla-action}
     \nabla ( \f_i^s ) = (z-z_i)^{-1} \left( \sum_{s' = 1}^r \tilde{a}_{i,s's}(z) \g_i^{s'}\right) \d z. 
\end{equation}
\begin{lem}
 The coefficients appearing on the right-hand side of (\ref{eq:nabla-action}) satisfy the following conditions.
 \begin{enumerate}
  \item $\tilde{a}_{i,s's}$ are regular at $z_i$ \label{item:vanishing1}
  \item if $s',s$ are such that $\mu_i^s - \mu_i^{s'} \notin \Z$ then $\tilde{a}_{i,s's} =0$ \label{item:vanishing2}
  \item if $s',s$ are such that $j(s') > j(s)$ then $\tilde{a}_{i,s's}(z_i) = 0$. \label{item:vanishing3}
 \end{enumerate}
\end{lem}
\begin{proof}
These claims can be directly checked using the definitions of $\f_i^s,\g_i^s$ and the connection form of $\nabla$ with respect to the vectors $\e_i^s$ given above. 
Indeed, the only coefficients $\tilde{a}_{i,s's}$ that may be more singular than the corresponding coefficient ${a}_{i,s's}$ correspond to pairs of indices $s',s$ 
such that $\g_i^{s'} = (z-z_i) \e_i^{s'}$ i.e. $\beta_i^{j(s')} = 0, k(s') \geq -1$. We separate cases. 

If $\beta_i^{j(s)} = 0, \mu_i^{s} = 0$ then by Assumption \ref{assn:main} (\ref{assn:main1}) 
we have that $\f_i^s = \e_i^s$ is annihilated by $\res_{z_i}^{j(s)}$, said differently $a_{i,s's} (z_i) =0$, 
hence $\tilde{a}_{i,s's} = (z-z_i)^{-1}a_{i,s's}$ is regular at $z_i$. 

If $\beta_i^{j(s)} = 0,\mu_i^{s} \neq 0$ and $k(s) < -1$ then $\f_i^s = \e_i^s$, and as the nilpotent part $N_i^{j(s)}$ 
decreases the index of the weight-filtration $k$, it follows that the coefficient $a_{i,s's}$ vanishes at $z_i$, 
hence again $\tilde{a}_{i,s's}$ is regular at $z_i$. 

If $\beta_i^{j(s)} = 0, \mu_i^{s} \neq 0$ and $k(s) \geq -1$ then $\f_i^s = (z-z_i) \e_i^s$, so 
$\tilde{a}_{i,s's} = a_{i,s's} + \delta_{s's}$, thus $\tilde{a}_{i,s's}$ is regular at $z_i$ since $a_{i,s's}$ is regular there. 

Finally, if $\beta_i^{j(s)} > 0$ then by Assumption \ref{assn:main} (\ref{assn:main2}) %NO INTEGER EIGENVALUES FOR POSITIVE WEIGHTS
and the vanishing condition on $a_{i,s's}$, we see that $a_{i,s's}$ is identically zero, therefore so is $\tilde{a}_{i,s's}$. 
This finishes the proof of part (\ref{item:vanishing1}).

Similarly, for all values of $s',s$ such that $\mu_i^s - \mu_i^{s'} \notin \Z$ we have 
$$
  \tilde{a}_{i,s's} = (z-z_i)^l a_{i,s's}
$$
for some integer $l$, hence $a_{i,s's} = 0$ implies $\tilde{a}_{i,s's} =0$. 
This finishes the proof of part (\ref{item:vanishing2}). 

As for part (\ref{item:vanishing3}), notice that the condition $j(s') > j(s)$ implies in particular that $\beta_i^{j(s')} > 0$. 
It then follows that  $\f_i^s = \e_i^s = \g_i^s$, and in particular $\tilde{a}_{i,s's} = a_{i,s's}$, for which the 
corresponding statement is known. 
\end{proof}

We return to the proof of the Proposition. 
For vectors $\e_i^{s}$ corresponding to $0$ parabolic weight and $0$ eigenvalue of the residue, 
by Assumption \ref{assn:main} (\ref{assn:main1})  %NO NON-ZERO INTEGER EIGENVALUE CORRESPONDS TO WEIGHT 0
and the above vanishing conditions only vectors of the same kind appear with non-zero 
coefficient on the right-hand side of (\ref{eq:nabla-action}). Given the definitions of $\F_{\beta}$, one then 
sees that the right-hand side belongs to $\G_{\beta}\otimes K_{\CP}(z_i)$. 

Fix now $1\leq s \leq r$ such that $\beta_i^{j(s)} > 0$ or $\beta_i^{j(s)} = 0 \neq \mu_i^{j(s)}$ and let $\beta \in \R$ be arbitrary. 
Recall that we set $m$ the unique integer satisfying $\beta_i^{j(s)} + m -1 < \beta \leq \beta_i^{j(s)} + m$, and we used 
$m$ to define $\f_i^s(\beta ) = (z-z_i)^m \f_i^s$. 
Let us now write an obvious consequence of the Lemma: if $\beta_i^{j(s)} > 0$ or $\beta_i^{j(s)} = 0 \neq \mu_i^{j(s)}$ then 
\begin{align*}
     \nabla ( (z-z_i)^m \f_i^s ) & =  (z-z_i)^{m-1} \left( m \f_i^s + \sum_{s' = 1}^r \tilde{a}_{i,s's}(z) \g_i^{s'}\right) \d z \\
				 & =  (z-z_i)^{m-1} \left( \sum_{s' = 1}^r b_{i,s's}(z) \g_i^{s'}\right) \d z 
\end{align*}
with 
$$
  b_{i,s's}(z) = \tilde{a}_{i,s's}(z) + m \delta_{s's}.
$$
(For ease of notation we do not write out the dependence of $b_{i,s's}$ on $m$). 
This formula readily follows from the definitions of $\f_i^s, \g_i^s$, as they only differ in the case 
$\beta_i^{j(s)} = 0 = \mu_i^{j(s)}$ that we exclude here. 
The Lemma then implies the same regularity and vanishing conditions for the coefficients $b_{i,s's}$ as for $\tilde{a}_{i,s's}(z)$. 
We will now show that all the vector components on the right-hand side of the above expression for $\nabla ( (z-z_i)^m \f_i^s )$ 
belong to $\G_{\beta}\otimes K_{\CP}(z_i)$. Let us introduce the subsets of indices 
\begin{align*}
   I_1 & = \{ s' | \quad \beta_i^{j(s')} + m -2 < \beta \leq \beta_i^{j(s')} + m - 1 \}, \\
   I_2 & = \{ s' | \quad \beta_i^{j(s')} + m -1 < \beta \leq \beta_i^{j(s')} + m \}, \\
   I_3 & = \{ s' | \quad \beta_i^{j(s')} + m < \beta \leq \beta_i^{j(s')} + m + 1 \}. 
\end{align*}
Then, it is easy to see that $I_1 \coprod I_2 \coprod I_3 = \{ 1, \ldots , n \}$, as all weights lie within an interval of length $1$. 
For $s'\in I_1$, we have $\g_i^{s'} (\beta ) = (z-z_i)^{m-1} \g_i^{s'}$, 
hence the corresponding term on the right-hand side belongs to $\G_{\beta}\otimes K_{\CP}$ because $b_{i,s's}$ is regular.
For $s'\in I_2$, we have $\g_i^{s'} (\beta ) = (z-z_i)^m \g_i^{s'}$, 
hence the corresponding term on the right-hand side belongs to $\G_{\beta}\otimes K_{\CP}(z_i)$ because $b_{i,s's}$ is regular. 
Finally, for $s' \in I_3$ we necessarily have $\beta_i^{j(s')} < \beta_i^{j(s)}$, said differently $j(s') > j(s)$, so 
we have $b_{i,s's}(z_i) = 0$. Moreover, in this case we have $\g_i^{s'} (\beta ) = (z-z_i)^{m+1} \g_i^{s'}$. 
Therefore, the corresponding term on the right-hand side is of the form 
$$
  (z-z_i)^{-2} b_{i,s's}(z) \g_i^{s'}(\beta ), 
$$
which is again in $\G_{\beta}\otimes K_{\CP}(z_i)$ as $b_{i,s's}(z_i) = 0$. This finishes the proof.
\end{proof}

By virtue of the proposition, we may consider the filtered versions of (\ref{eq:nabla-extension}) that read as: 
\begin{equation}\label{eq:filtered-nabla-extension}
   \nabla_{\beta, \zeta} = \nabla_{\beta} \otimes s_{\infty} - \Id_{\E}\d z \otimes s_0 : \pi_1^*\F_{\beta} \to \pi_1^*\G_{\beta} \otimes K_{\CP}(2\cdot z_0 + z_1+\cdots + z_n) \otimes \pi_2^*  \O_{\CPt}(1). 
\end{equation}
For convenience, we introduce the notation $\dR_{\beta, \zeta}$ for the above complex, that we call the twisted filtered 
de Rham complex. Analogously to (\ref{eq:transformed_bundle}), for all $\beta \in [ 0,1 )$ we may then set 
\begin{equation}\label{eq:transformed_parabolic_sheaf}
   \Et_{\beta} = \R^1 (\pi_2)_* \dR_{\beta, \bullet} . 
\end{equation}
%the hypercohomology space of degree $1$ of $\dR_{\beta, \zeta}$.

\begin{clm}\label{clm:qis}
For all $\zeta \notin \Pt \cup \{\infty \}$ and any $\beta, \beta'\in \R$ sufficiently close to each other, 
the filtered twisted de Rham complexes $\dR_{\beta, \zeta}$ and $\dR_{\beta', \zeta}$ are quasi-isomorphic.
In particular, for all $\beta \in [0,1)$ the sheaf $\Et_{\beta}$ is a locally free subsheaf of $\Et_0$ of full rank. 
\end{clm}

\begin{proof}
 Without loss of generality we may assume that $\beta > \beta'$. 
 Writing $L = K_{\CP}(2\cdot z_0 + z_1+\cdots + z_n)$ we then have a short exact sequence of complexes of sheaves of $\C$-vector spaces on $\CP$ 
 $$
  \xymatrix{\F_{\beta} / \F_{\beta'} \ar[r] & (\G_{\beta} / \G_{\beta'})\otimes L \\
	    \F_{\beta}  \ar[r]^{\nabla_{\beta, \zeta}}  \ar[u] & \G_{\beta}\otimes L \ar[u]\\
	    \F_{\beta'} \ar[r]^{\nabla_{\beta', \zeta}}  \ar[u] & \G_{\beta'}\otimes L \ar[u]} 
 $$
 with the upper horizontal map induced by the middle one. 
 By the definition of $\F_{\beta}, \G_{\beta}$ they are equal to $\E$ away from the set $P \cup \{\infty \}$, therefore the sheaves in the top row of this diagram are 
 supported at this finite set. 
 Now assume that for all $1 \leq i \leq n$ there is at most one value of the form $\beta_i^j + m$ between $\beta$ and $\beta'$. 
 If there is zero such value between $\beta$ and $\beta'$ then near $z_i$ the complexes $\dR_{\beta, \zeta}$ and $\dR_{\beta', \zeta}$ are clearly quasi-isomorphic near $z_i$. 
 Let us now assume that there exists exactly one such  $\beta \geq \beta_i^j + m > \beta'$. 
 Then, there are two possibilities: either $\beta_i^j = 0$ or $\beta_i^j > 0$. 
 If $\beta_i^j = 0$ then it follows from the definitions of $\F_{\beta}, \G_{\beta}$ that the quotients $\F_{\beta} / \F_{\beta'}$ do not contain any classes represented by 
 vectors of the form $\phi e_i^s$ such that $\mu_i^s = 0$, for any non-zero function $\phi$. 
 On the other hand, for all other basis vectors  it follows from Assumptions \ref{assn:main} (\ref{assn:main1}) and (\ref{assn:main2}) % nem lehet beta = 0 eseten 0-n kivul egesz sajatertek; nem lehet beta > 0 eseten egesz sajatertek
 coupled with (\ref{eq:nabla-action}) that on the quotient $\F_{\beta} / \F_{\beta'}$ the action of $\nabla_{\beta, \zeta}$ is by non-zero constants. 
 Therefore, the morphism induced on the quotients maps the fiber of $\F_{\beta} / \F_{\beta'}$ over ${z_i}$ isomorphically onto that of $(\G_{\beta} / \G_{\beta'})\otimes L$. 
 
 For $i=0$ from the assumption $\zeta \notin \Pt$, the formulas (\ref{eq:Dinfinity}), (\ref{eq:twisteddRmap}) 
 and compatibility of $\nabla$ with the parabolic structure, we see that locally the map $\nabla_{\beta, \zeta}$ maps the fiber of $\F_{\beta}$ over $z_0$ 
 isomorphically onto that of $(\G\otimes L)_{\beta}$. This will then clearly remain the case after passing to quotients. 
 
 Now, it follows from the first statement that $\Et_{\beta}$ is a subsheaf of full rank. 
 As $\Et_0$ is locally free by Proposition \ref{prop:locally_free} and $\CP$ is smooth, we obtain the desired result. 
\end{proof}

We now set
\begin{equation}
 \widehat{D}_{\red} = \mbox{div} (\Pt ) + \infty ,
\end{equation}
where $\mbox{div}$ stands to denote the simple effective divisor associated to a set of distinct points, 
and extend the definition of $\Et_{\beta}$ to all $\beta \in \R$ by the requirement 
$$
 \Et_{\beta + 1} = \Et_{\beta} (-\widehat{D}_{\red} ). 
$$

\begin{prop}\label{prop:par_str}
Assume that in addition to Assumption \ref{assn:main} the following conditions are also fulfilled for all $i,j$: 
 \begin{itemize}
  \item 
   if $i=0$ then $\beta_i^j$ is not an eigenvalue of $\res_{z_i}(\nabla )^j$;
  \item 
   if $i > 0$ and %for some $s$ with $j(s) =j$ the real part $\Re \mu_i^s$ of the corresponding eigenvalue of $\res_{z_i}(\nabla )^j$ vanishes 
   $\beta_i^j = 0$ 
   then the nilpotent part $N_i^j$ of the residue acts trivially on the generalized $0$-eigenspace of $\res_{z_i}(\nabla )^j$; 
 \item 
   if $i > 0$ and for some $s$ with $j(s) =j$ we have $\Re \mu_i^s \neq 0$ then $\beta_i^j \neq \mu_i^s$. %is not an eigenvalue of $\res_{z_i}(\nabla )^j$.
  \end{itemize}
Then, the family $\{ \Et_{\beta} \}_{\beta \in \R}$ is an $\R$-parabolic sheaf with divisor $\widehat{D}_{\red}$. 
\end{prop}
\begin{rk}\label{rem:Assumptions}
 According to Simpson's table \cite{Sim} p. 720, the parabolic weights and eigenvalues of the Higgs bundle associated to $(V, F_i^j, \dbar^{\E}, \nabla , h)$ 
 read respectively as 
 $$
  \alpha_i^s = \Re \mu_i^s, \quad \lambda_i^s = \frac{\mu_i^s - \beta_i^{j(s)}}2.
 $$
 The assumptions listed in the statement of Proposition \ref{prop:par_str} then mean that the Assumptions of \cite{Sz-FM34} hold for the 
 associated Higgs bundle. 
\end{rk}

\begin{proof}
%Endow the vector bundle $F_{1-\varepsilon}$ with the parabolic filtration induced by the one of $F_0$ and parabolic weights 
For some small $\varepsilon$ the Hermitian metric $h$ is the harmonic metric for $F_{1-\varepsilon}$ and the underlying holomorphic vector bundle 
of the associated Higgs bundle is $\mathcal{F}_{1-\varepsilon}$ appearing in Section 5 \cite{Sz-FM34}. 
In particular, the first hypercohomology spaces of the twisted de Rham and Dolbeault complexes both compute the space of $L^2$ harmonic $1$-forms with 
values in $V_{1-\varepsilon}$, so they are isomorphic. 
It then follows that the smooth vector bundles underlying the holomorphic vector bundles 
$$
  \Et_{1-\varepsilon}  \quad \mbox{and} \quad \widehat{\mathcal{E}}_{1-\varepsilon}
$$
agree, as the fiber over $\zeta$ of the former is $\H^1 (\dR_{1-\varepsilon, \zeta} )$ and the one of the latter is $\H^1 (\mbox{Dol}_{1-\varepsilon, \zeta} )$. 
According to Proposition 5.6 \cite{Sz-FM34}, we have an inclusion of sheaves of $\O$-modules 
$$
   \widehat{\mathcal{E}}_0 \otimes \O_{\CPt}(-\widehat{D}_{\red}) \subseteq \widehat{\mathcal{E}}_{1-\varepsilon} .
$$
It then follows that a similar inclusion holds for $\Et$ too. 
\end{proof}

\section{Minimal Fourier--Laplace transformation}\label{sec:Laplace}

In this section, we recall the classical construction of minimal Fourier--Laplace transformation of a holonomic $\mathcal{D}_{\CP}$-module. 
Throughout the section we use the following standard notation: given a complex analytic manifold $X$, a smooth divisor $D\subset X$ 
and a locally free sheaf $S$ of $\O_X$-modules over $X$, we denote by $S(*D)$ the sheaf of meromorphic sections of $S$ with poles 
of arbitrary order along $D$. 

\subsection{Minimal extension of a $\mathcal{D}_{C}$-module}

For a complex curve $C$ we let $\mathcal{D}_{C}$ stand for the sheaf of analytic differential operators on $C$. 
Locally, in some complex chart $w$ the sections of $\mathcal{D}_{C}$ are of the form 
$$
  \sum_i h_i(w) (\partial_w)^i 
$$
where the sum is finite and $h_i$ are analytic functions.
Given a vector bundle $E$ endowed with an integrable connection $\nabla$ with singularities in the points $z_i$ as in Section \ref{sec:harmonic-bundles}, 
a classical construction associates a left $\mathcal{D}_{\CP}$-module $\mathcal{M}$ to $(E, \nabla )$: as an $\O_{\CP}$-module, 
$\mathcal{M}$ is isomorphic to the meromorphic bundle 
$$
  \mathcal{M} = E (* (z_0 + z_1 + \cdots + z_n)),
$$ 
with the differential operator $\partial_z$ acting on a section $e \in E(U)$ on some open set $U \subseteq C^0$ by 
$$
  \partial_z (e) = \nabla (e) \left( \frac{\partial}{\partial z} \right). 
$$
We say that a $\mathcal{D}_{\CP}$-module $\mathcal{M}$ of this type is a meromorphic bundle. 
Meromorphic bundles are holonomic $\mathcal{D}_{\CP}$-modules: every local section is annihilated by a local section of $\mathcal{D}_{\CP}$. 

We define an extension $\mathcal{N}$ of $\mathcal{M}$ as a $\mathcal{D}_{\CP}$-module such that the meromorphic bundle 
$$
  \mathcal{N} \otimes_{\O_{\CP}} \O_{\CP} (* (z_0 + z_1 + \cdots + z_n))
$$ 
is isomorphic to $\mathcal{M}$ as a $\mathcal{D}_{\CP}$-module. 
Observe that for any $\mathcal{M}$ of the above type (i.e. a meromorphic bundle) $\mathcal{N} = \mathcal{M}$ is always an extension. 
An extension $\mathcal{N}$ is called minimal if it has no non-trivial submodules and no non-trivial quotient modules. 
\begin{example}\label{ex:minext}
In these examples we let $\mathcal{M} = \O_{\CP} (* \{ 0 \})$ as an $\O_{\CP}$-module. 
\begin{enumerate}
 \item \label{ex:minext1} If $\partial_z$ acts by the trivial connection 
 $$
  \partial_z \cdot f = \frac{\d f}{\d z} 
 $$
 then $\mathcal{N} = \mathcal{M}$ is not a minimal extension: indeed, then the submodule 
 $$
  \O_{\CP} \subset \O_{\CP} (* \{ 0 \})
 $$
 is preserved by $\partial_z$. Instead, the lattice $\O_{\CP}$ with the induced $\partial_z$-action is a minimal extension. 
 \item \label{ex:minext2} Similarly, if $\partial_z$ acts by a logarithmic connection 
 $$
  \partial_z \cdot f = \left( \frac{\d}{\d z} + \frac nz \right) f 
 $$
 with integer residue $n \in \Z$ at $0$ then $\mathcal{N} = \mathcal{M}$ is not a minimal extension, because of the 
 submodule $\O_{\CP} (n \cdot \{ 0 \})$. Again, this latter lattice with the induced $\partial_z$-action is a minimal extension. 
 \item \label{ex:minext3} On the other hand, if we let $\partial_z$ act by a logarithmic connection 
 $$
  \partial_z \cdot f = \left( \frac{\d}{\d z} + \frac{\mu}z \right) f 
 $$
 with some non-integer residue $\mu \in \C \setminus \Z$ at $0$, then $\mathcal{N} = \mathcal{M}$ is a minimal extension. 
 Indeed, any non-trivial submodule of $\mathcal{N}$ is a locally free $\O_{\CP}$-module, and is necessarily of 
 the type $\O_{\CP} (n \cdot \{ 0 \})$ for some $n\in \Z$; however, it is easy to see that such modules 
 are not preserved by $\partial_z$. 
 \item \label{ex:minext-irregular} If $\partial_z$ acts by a connection with an irregular singularity 
 $$
  \partial_z \cdot f = \left( \frac{\d}{\d z} + m(z) \right) f 
 $$
 where $m$ is a meromorphic function with a pole of order at least $2$ at $0$, then $\mathcal{M}$ is a minimal extension. 
 Indeed, if the $z$-adic valuation of $f$ is some $n \in \Z$ then the $z$-adic valuation of $\partial_z \cdot f$ 
 is necessarily less than or equal to $n-2$, so again no submodule of the form $\O_{\CP} (n \cdot \{ 0 \})$ may be 
 invariant by $\partial_z$. 
\end{enumerate}
\end{example}

\begin{example}\label{ex:minext-higher-rank}
 Let us now take the $r$-fold direct power $\O_{\CP}$-module 
 $$
  \mathcal{M} = \O_{\CP} (* \{ 0 \}) \cdot e^1 \oplus \cdots \oplus \O_{\CP} (* \{ 0 \}) \cdot e^r
 $$
 for some $r\geq 2$.
 \begin{enumerate}
 \item 
 Let $\partial_z$ act on $\mathcal{M}$ by a logarithmic connection with maximal nilpotent residue at $0$, i.e. 
 $$
  \partial_z e^1 = 0, \quad \partial_z e^j = \frac{e^{j-1}}z \mbox{ for } j\geq 2. 
 $$
 Then a minimal extension is given by the $\O_{\CP}$-module 
 $$
  \O_{\CP} (* \{ 0 \}) \cdot e^1 \oplus \cdots \oplus \O_{\CP} (* \{ 0 \}) \cdot e^{r-1} \oplus \O_{\CP} \cdot e^r. 
 $$
 For simplicity, we sketch the argument in the case $r=2$. As usual, denote by $\O_0$ and $K_0$ the ring of local functions at $0\in\CP$ and its 
 fraction field. Assume there exist some elements $a_1,b_1 \in K_0$ and $a_2,b_2 \in \O_0$ such that the $\O_{\CP}$-module generated by 
 \begin{equation}\label{eq:generators}
  a_1 e^1 + a_2 e^2 \quad \mbox{and} \quad  b_1 e^1 + b_2 e^2
 \end{equation}
 is $\partial_z$-invariant. The action of $\partial_z$ reads as 
 \begin{equation}\label{eq:partial-action}
    \partial_z (a_1 e^1 + a_2 e^2) = \left( \mbox{d} a_1 + a_2 \frac{\mbox{d}z}z \right) e^1 + \mbox{d} a_2 e^2, 
 \end{equation}
 and similarly for $a_j$ replaced by $b_j$. 
 With respect to the local coordinate $z$ near $0\in\CP$, we may speak of the pole order of $a_1,b_1$. 
 Assume that at least one of $a_1,b_1$ has a pole of order at least $1$. Assume without loss of generality that 
 the order of the pole of $a_1$ is at least as high as that of $b_1$. Then, the order of the pole of $\mbox{d} a_1$ 
 is strictly larger than the maximum of the pole orders of $a_1,b_1$ and the term $a_2\mbox{d}z/z$ does not change the 
 pole order of the coefficient of $e^1$ on the right-hand side of (\ref{eq:partial-action}), 
 hence (\ref{eq:partial-action}) may not be expressed as a linear combination of $a_1,b_1$ with coefficients in $\O_0$. 
 We thus see that $a_1,b_1 \in \O_0$. Looking at the coefficient of $e^1$ in (\ref{eq:partial-action}) we also see that 
 one must then have $a_2(0) = 0$, and similarly we get $b_2(0) = 0$. Now, the coefficient of $e^2$ in (\ref{eq:partial-action}) is a 
 linear combination of $a_2,b_2$ with coefficients in $\O_0$, hence we get $\mbox{d} a_2 (0) = 0$ and similarly $\mbox{d} b_2 (0) = 0$. 
 It is now easy to see by induction that $a_2, b_2$ vanish to order $k$ for any $k \in \N$. Said differently, $a_2 = 0 = b_2$, 
 and $\mathcal{M}$ is not generated by (\ref{eq:generators}) as an $\O_{\CP} (* \{ 0 \})$-module, a contradiction. 
 \item Let $\partial_z$ act on $\mathcal{M}$ by a logarithmic connection with a regular residue with a single integer eigenvalue $n\in \Z$ at $0$, i.e. 
 up to conjugacy a Jordan block: 
 $$
  \partial_z e^1 = \frac nz e^1, \quad \partial_z e^j = \frac{n e^j + e^{j-1}}z \mbox{ for } j\geq 2. 
 $$
 Then, combining the ideas of Example \ref{ex:minext} (\ref{ex:minext2}) with the above, it is easy to see that a minimal extension is given by 
 $$
  \O_{\CP} (* \{ 0 \}) \cdot e^1 \oplus \cdots \oplus \O_{\CP} (* \{ 0 \}) \cdot e^{r-1} \oplus  \O_{\CP}(n \{ 0 \}) \cdot e^r. 
 $$
  \item Let now $\partial_z$ act on $\mathcal{M}$ by a logarithmic connection with a regular residue with a single non-integer eigenvalue 
  $\mu \in \C \setminus \Z$ at $0$, i.e.
 $$
  \partial_z e^1 = \frac{\mu}z e^1, \quad \partial_z e^j = \frac{\mu e^j + e^{j-1}}z \mbox{ for } j\geq 2. 
 $$
 Then, just as in Example \ref{ex:minext} (\ref{ex:minext3}), a minimal extension is given by $\mathcal{M}$. 
\end{enumerate} 
\end{example}

It follows from Section 1 of \cite{Sab-hi} that for any meromorphic bundle $\mathcal{M}$ a minimal extension 
exists and is unique up to isomorphism; we denote it by $\mathcal{M}_{\min}$. 
The above Examples give the local form of the minimal extension of a meromorphic bundle $\mathcal{D}_{\CP}$-module $\mathcal{M}$ induced by an integrable connection with poles. 
In particular, according to our assumption (\ref{eq:Dinfinity}) and Example \ref{ex:minext} (\ref{ex:minext-irregular}), the minimal extension of the meromorphic bundle $\mathcal{M}$ 
near $\infty$ is equal to $\mathcal{M}$ itself. 
On the other hand, for all $1\leq i \leq n$ and arbitrary $\mu \in \{ \mu_i^s \}_{s=1}^r$ let 
$$
  \psi_i^{\mu} = \ker(\res_{z_i}(\nabla ) - \mu)^r
$$
denote the generalized $\mu$-eigenspace of $\res_{z_i}(\nabla )$. 
We will use the convention that $\psi_i^{\mu} = 0$ for any $\mu \notin \{ \mu_i^s \}_{s=1}^r$. 
Then the above examples show that the minimal extension has an explicit description in terms 
of the kernel and cokernel of the restriction of the residue to the various eigenspaces $\psi_i^{\mu}$. 
This observation will play a key role in the proof of Theorem \ref{thm:main}. 

\subsection{Fourier--Laplace transformation}

Consider the one-dimensional Weyl algebra ${\C}[z]\langle\partial_{z}\rangle$ generated by 
the formal variables $z$ and $\partial_{z}$ subject to the only relation $[z,\partial_{z}]=-1$. 
Let us be given a left ${\C}[z]\langle\partial_{z}\rangle$-module $M$. 
Furthermore, let $\zeta$ be a variable algebraically independent of $z$ and consider the 
Weyl algebra ${\C}[\zeta ]\langle\partial_{\zeta}\rangle$ defined as above, with each occurence of 
$z$ replaced by $\zeta$. We then define the Fourier--Laplace transform $\widehat{M}$ of $M$ as follows: 
as a $\C$-vector space, we set $\widehat{M} = M$, and we let $\zeta,\partial_{\zeta}$ act by 
\begin{align*}
 \zeta \cdot m & = \partial_{z} \cdot m \\
 \partial_{\zeta} \cdot m & = - z \cdot m. 
\end{align*}
It can be easily checked that this then turns $\widehat{M}$ into a left ${\C}[\zeta ]\langle\partial_{\zeta}\rangle$-module. 

Let us now give an equivalent description of the Fourier--Laplace transform $\widehat{M}$ of $M$ in terms of the 
cokernel of a suitable map of modules, according to \cite{KL} Lemma 7.1.4. 
For this purpose, let us consider the two-dimensional Weyl algebra 
${\C}[z, \zeta]\langle\partial_{z}, \partial_{\zeta}\rangle$ generated by the formal variables 
$z,\zeta, \partial_{z}, \partial_{\zeta}$ subject to the relations 
$$                                                                                                                                                 
 [z,\partial_{z}]  = - 1 \quad [\zeta,\partial_{\zeta}]  = - 1, 
$$
an all other generators commuting with each other. 
We consider the module 
$$
    \mathbf{M} = M \otimes_{\C} \C[\zeta].
$$
We turn $\mathbf{M}$ into a left module over ${\C}[z, \zeta]\langle\partial_{z}, \partial_{\zeta}\rangle$ as follows: 
$z,\partial_{z}$ act on $M$ and $\zeta,\partial_{\zeta}$ act on $\C[\zeta]$ in the standard way. 
It then turns out that the kernel of the map 
$$
  \partial_{z} - \zeta : \mathbf{M} \to \mathbf{M}
$$
vanishes, and its cokernel is isomorphic to $M$ as a $\C$-vector space. 
Furthermore, this cokernel space inherits a left ${\C}[\zeta ]\langle\partial_{\zeta}\rangle$-module structure 
from $\mathbf{M}$, with the action of $\partial_{\zeta}$ induced by the action of  
$$
  \partial_{z} + \partial_{\zeta} - z - \zeta 
$$
on $\mathbf{M}$. An easy argument shows that the cokernel space endowed with this 
left ${\C}[\zeta ]\langle\partial_{\zeta}\rangle$-module structure is isomorphic to $\widehat{M}$.

\subsection{$\mathcal{D}_{\CP}$-modules and ${\C}[z]\langle\partial_{z}\rangle$-modules}

To any left $\mathcal{D}_{\CP}$-module $\mathcal{M}$ we may associate a left module $M$ over 
${\C}[z]\langle\partial_{z}\rangle$ by the following globalization procedure: we let 
$$
  M = \Gamma (\C, \mathcal{M})
$$
be the $\C$-vector space of sections of $\mathcal{M}$ over $\C$, we let $z$ act on such a section 
by multiplication by $z$, and $\partial_{z}$ act on $M$ via the action of $\partial_{z}$ on $\mathcal{M}$. 

Conversely, given any left ${\C}[z]\langle\partial_{z}\rangle$-module $M$, we may take the sheafification $\mathcal{M}$ of $M$. 
Namely, we may define $\mathcal{M}$ as the sheaf of $\O (\CP )$-modules associated to the preseheaf in analytic topology 
$$
  U \mapsto M \otimes_{\O (\CP )} \O (U), 
$$
and we may let $\partial_z$ act on $\mathcal{M}$ by 
$$
  \partial_z \cdot (f m) = f \partial_z \cdot  m + \frac{\d f}{\d z} m 
$$
for any $f\in \O (U), m \in \mathcal{M} (U)$. Then, $\mathcal{M}$ is a left $\mathcal{D}_{\CP}$-module.

\subsection{Minimal Fourier--Laplace transformation}

For any integrable bundle $(E, \nabla )$ with singularities as in Section \ref{sec:harmonic-bundles}, 
consider the minimal extension $\mathcal{M}_{\min}$ associated to the meromorphic bundle $\mathcal{M}$,  
and denote by $M_{\min}$ the corresponding left ${\C}[z]\langle\partial_{z}\rangle$-module. 
Consider moreover the Fourier--Laplace transform $\widehat{M_{\min}}$, 
and denote the associated left $\mathcal{D}_{\CP}$-module by $\widehat{\mathcal{M}_{\min}}$. 
We then define the minimal Fourier--Laplace transform of $\mathcal{M}_{\min}$ as $\widehat{\mathcal{M}_{\min}}$.  

In the same way as in Proposition 4.2 \cite{Sz-laplace}, Assumption \ref{assn:main} (\ref{assn:main0}) implies the following result. 
\begin{prop}\label{prop:min-FL}
 The $\mathcal{D}_{\CPt}$-module $\widehat{\mathcal{M}_{\min}}$ is the minimal extension of a meromorphic integrable bundle $\widehat{\mathcal{M}}$ 
 with singularities at the set $\Pt$ of eigenvalues of the leading-order term $A$ of $\nabla$ at infinity. 
 In particular, the minimal Fourier--Laplace transform of $\widehat{\mathcal{M}_{\min}}$ is isomorphic to 
 $(-1)^*\mathcal{M}_{\min}$, where $(-1):\C \to \C$ is the map $z \mapsto -z$. 
\end{prop}

\begin{proof}
 Any sub- or quotient module of $\widehat{\mathcal{M}_{\min}}$ supported at $\zeta_l \in \Pt$ would necessarily come from a 
 sub- or quotient module of $\mathcal{M}$ with $\partial_z$-action induced by a connection of the form $\d - \zeta_l \d z$. 
 As already shown in Proposition \ref{prop:hypercohomology}, Assumption \ref{assn:main} (\ref{assn:main0}) excludes 
 the existence of such a sub- or quotient bundle. A similar argument works for $\zeta = \infty$. 
\end{proof}

\begin{thm}\label{thm:main}
 The meromorphic integrable bundle $\widehat{\mathcal{M}}$ appearing in Proposition \ref{prop:min-FL} is the meromorphic 
 integrable bundle associated to the Nahm-transform $(\Et, \widehat{\nabla})$. 
\end{thm}

\begin{proof}
 The key point is to show the following statement.

 \begin{prop}
  For all $\zeta \in \Ct \setminus \Pt$, the complex (\ref{eq:nabla-extension}) is quasi-isomorphic to 
  \begin{equation}\label{eq:minimal-dR}
       \mathcal{M}_{\min} \xrightarrow{\partial_z - \zeta} \mathcal{M}_{\min}. 
  \end{equation}
 \end{prop}
 \begin{proof}
 In view of Assumptions \ref{assn:main} (\ref{assn:main1}) and (\ref{assn:main2}), %NINCS NEM-0 EGESZ SAJATERTEK 0 PAR SULLYAL & NINCS EGESZ SAJATERTEK NEM-0 PAR SULLYAL
 the formula for $\mathcal{M}_{\min}$ given in Examples \ref{ex:minext} and \ref{ex:minext-higher-rank} is quite simple. 
 Indeed, the only integer eigenvalue appearing in $A_i$ is $\mu_i^s = 0$, and by Assumption \ref{assn:main} (\ref{assn:main1}) %0 SAJATERTEKHEZ TARTOZO NILPOTENS RESZ 0
 we have 
 $$
  \ker (\res_{z_i}(\nabla ) |_{ \psi_i^{0}}) = \psi_i^{0}, \quad \im (\res_{z_i}(\nabla ) |_{ \psi_i^{0}} ) = 0. 
 $$
 We infer that 
 $$
  \mathcal{M}_{\min} = \bigoplus_{i=1}^n \left( \bigoplus_{\mu \neq 0 } \psi_i^{\mu} \otimes_{\C} \C\{ z-z_i \}[(z-z_i)^{-1}]  \right) \oplus  %\O_{\CP} (* \{ z_i \}) \right) \oplus 
   \left( \psi_i^{0} \otimes_{\C} \C\{ z-z_i \} \right) %\O_{\CP} \right)
 $$
 as a ${\C}[z]\langle\partial_{z}\rangle$-module. 
 Now, it follows from the definition of $\F_{\beta}, \G_{\beta}$ in Subsection \ref{ssec:transforming-parabolic} and 
 Assumption \ref{assn:main} (\ref{assn:main2}) % NINCS 0 SAJATERTEK NEM-0 PAR SULLYAL
 that the quasi-coherent $\O_{\CP}$-module $\mathcal{M}_{\min}$ is the inductive limit of the coherent $\O_{\CP}$-modules $\F_{\beta}, \G_{\beta}$ 
 as $\beta \to \infty$. 
 Claim \ref{clm:qis} shows that the corresponding filtered twisted de Rham complexes $\dR_{\beta, \zeta}$ are all 
 quasi-isomorphic to each other. 
 As the complex (\ref{eq:minimal-dR}) is an inductive limit of the complexes $\dR_{\beta, \zeta}$, and these latter 
 are all quasi-isomorphic to each other, we infer that (\ref{eq:minimal-dR}) is quasi-isomorphic to $\dR_{\beta, \zeta}$ 
 for any $\beta \in \R$. This then holds in particular for $\beta = 0$. 
 \end{proof}
With the proposition established, the line of argument of Section 5 of \cite{Sz-laplace} proves the theorem. 
\end{proof}

Just as in Theorem 4.1 \cite{Sz-laplace}, the above theorem has the following consequence. 
\begin{cor}\label{cor:dR-isom} 
 The transform (\ref{eq:Nahm}) preserves the de Rham complex structure of the moduli spaces. 
 %The rank of $\Et$ is given by 
 %\begin{equation}\label{eq:transfomed_rank}
 % \hat{r} = \rank (\Et ) = \sum_{i=1}^n \sum_{\mu \neq 0 } \dim_{\C} \psi_i^{\mu}
 %\end{equation}
\end{cor}

%\begin{proof}
 %The first statement is immediate as Theorem \ref{thm:main} gives an algebraic interpretation of the transform. 
 %As for the formula for the rank, observe first that for any $\mu \neq 0$ the operator $\partial_z - \zeta$ acts with no 
 %kernel on $\psi_i^{\mu} \otimes_{\C} \C\{ z-z_i \}[(z-z_i)^{-1}]$, second that for any $v \in \psi_i^{0}$ it annihilates $e^{z\zeta} v$. 
%\end{proof}

\section{Transformation of the singularity parameters}\label{sec:transformation}

In this section we describe the transformation of the singularity parameters under $\Nahm$ in the most simple situation where the 
residues of the associated Higgs bundle are regular and satisfy the assumptions made in \cite{Sz-FM34}. 
%and illustrate the complications that arise when these latter assumptions are lifted. 
We plan to return to the case of non-necessarily regular residue in ongoing joint work with Takuro Mochizuki. 
%We find results analogous to the stationary phase formula for $\ell$-adic Fourier transform of Katz and Laumon \cite{KL}. 

\subsection{Irregular singularity of the transformed Higgs bundle}\label{subsec:irreg_st_phase}
We first treat the case of $\infty \in \CPt$. 
We denote by $J_s(\mu)$ a Jordan block of dimension $s$ for the eigenvalue $\mu$. 
For ease of notation, in the next theorem we assume $n=1$ and lift the subscript $i$ wherever applicable. 
\begin{thm}\label{thm:st_phase1}
Assume that the conditions of Proposition \ref{prop:par_str} are satisfied, and in addition that for any $1 \leq s\neq s' \leq r$ 
the following conditions hold: 
\begin{enumerate}
 \item  $\mu^s - \beta^{j(s)} = \mu^{s'} - \beta^{j(s')}$ implies $\beta^{j(s)} = \beta^{j(s')}$ 
    (and then necessarily $\mu^s = \mu^{s'}$ too);
 \item the graded residue $\res_{z_1}(\nabla )^j$ is regular, i.e. has Jordan decomposition 
  \begin{equation}\label{eq:Jordan}
   \bigoplus_m J_{s_m} (\mu^{j,m} ), 
  \end{equation}
    where $\mu^{j,m} \neq \mu^{j,l}$ for $m\neq l$, and $s_m \in \N$. 
\end{enumerate}
Then $\Dt^{(1,0)} + z_1 \d \zeta$ has a logarithmic singularity at $\zeta = \infty$ with respect to the extension of $\Et$ described 
in Section \ref{sec:transformed-parabolic}, the parabolic structure is compatible, and generically the $j^{th}$ graded piece of 
its residue (corresponding to the weight $\beta^j$) has the Jordan decomposition 
$$
  \bigoplus_{m|\mu^{j,m} \neq \beta^{j}}  J_{s_m} (-\mu^{j,m} ).
$$
\end{thm}
For  the generalization of the theorem to the case $n\geq 2$, we would need to make the same assumptions as in the case $n=1$ at each of the 
logarithmic points, and the conclusion should be replaced by saying that the transformed flat connection is a direct sum of flat 
connections behaving as in the case $n=1$, up to holomorphic terms. 
The proof of this generalization would then be similar to the case $n=1$, but with more complicated notation, hence we content 
ourselves with proving the theorem in its form as stated above. 
\begin{proof}
The idea is to use non-Abelian Hodge theory to switch to the Dolbeault interpretation of the moduli space, 
and the proof generalizes that of Theorem 4.31 \cite{Sz-these} where the case of regular semi-simple residue was treated. 
Namely, according to Theorem 1 \cite{Biq-Boa}, for every parabolically stable irregular flat connection there exists a unique harmonic metric, 
and therefore an irregular Higgs bundle $(\mathcal{E}, \theta)$ where $\mathcal{E}$ is a holomorphic vector bundle over $\CP$ and $\theta$ is a 
global meromorphic $1$-form valued endomorphism of $\mathcal{E}$, satisfying some properties. The moduli space of irregular Higgs 
bundles is often called irregular Dolbeault space, and the meromorphic integrable connection interpretation is referred to as de Rham space. 
Furthermore, it is shown in \cite{Biq-Boa} that the polar parts of the integrable connection and $\theta$ are related by the following transformation: 
\begin{itemize}
 \item the irregular part of the integrable connection $\nabla$ is equal to twice the irregular part of $\theta$; 
 \item the eigenvalues of the residue and the parabolic weights on the de Rham and Dolbeault sides obey the relations given by Simpson's table 
 (see Remark \ref{rem:Assumptions}); 
 \item the Jordan decompositions of the nilpotent parts of the residue of the connection and of the Higgs field admit blocks of the same size. 
\end{itemize}
%In principle, these latter transformations may allow for a non-zero eigenvalue of the residue on the de Rham side with 
%vanishing corresponding eigenvalue on the Dolbeault side. As we will see below, the vanishing of an eigenvalue of the residue of the Higgs field 
%with trivial nilpotent part leads to an unforeseen drop in the rank of the transformed bundle. 
%To exclude this phenomenon, in this section we therefore make a further genericity assumption: 
%for all $1\leq i \leq n$ and $1\leq s \leq r$ we have either 
%$$
%  \mu_i^s \neq \beta_i^s
%$$
%or 
%$$
%  \mu_i^s = \beta_i^s = 0. 
%$$
%Said differently, we assume that the corresponding eigenvalue of the Higgs field $\lambda_i^s = (\mu_i^s - \beta_i^s)/2$ does not vanish unless $\mu_i^s = \beta_i^s = 0$. 
In particular, in view of this correspondence, the assumptions of the theorem imply that the entire residue at $z_1$ of the Higgs bundle associated to $(E, \nabla )$ 
(rather than just its $j^{th}$ graded piece) is regular: indeed, any graded piece is regular by the second assumption and the eigenvalues of different graded 
pieces are different by the first assumption and Simpson's transformation table.  

By Proposition 5.6 \cite{Sz-FM34} (see also Proposition 4.22 \cite{Sz-these}), %and analogously to the extension given by (\ref{eq:nabla-extension}), 
in order to construct the transformed Higgs bundle we need to consider the spectral curve 
$$
  \Sigma = (\theta \xi - \mbox{Id}_E \zeta / 2) \subset H = \mathbf{P} (\O_{\CP} \oplus K_{\CP}(2\cdot z_0 + z_1+\cdots + z_n) )
$$
of $\theta$ in the Hirzebruch surface $H$, for the canonical sections $\xi,\zeta$ of 
$$
  \O_{\mbox{rel}}(1), K_{\CP}(2\cdot z_0 + z_1+\cdots + z_n) \otimes \O_{\mbox{rel}}(1)
$$
respectively, and the spectral sheaf 
$$
  M = \coker (\theta \xi - \mbox{Id}_E \zeta / 2)
$$
supported on $\Sigma$. Then $M$ is a torsion-free sheaf, pure of dimension $1$, and of rank $1$ for generic $(\mathcal{E}, \theta)$. 
We will use the birational transformation 
$$
  H \xleftarrow{\omega} H^+ \xrightarrow{\eta} \CP \times \CPt 
$$
where 
\begin{itemize}
 \item $\omega$ is the blow-up of the intersection points of the fibers of $H$ over the points $z_1, \ldots , z_n$ and the $0$-section $\zeta = 0$ of $H$
 \item $\eta$ is the blow-up of the points $(z_i,\infty) \in \CP \times \CPt$. 
\end{itemize}
In Section 5 of \cite{A-Sz} (see also \cite{Sz-BNR}), under some conditions we defined a proper transform functor $\omega^{\sigma}$ for coherent sheaves 
with respect to a blow-up map $\omega$, such that the support of $\omega^{\sigma}(M)$ is equal to the proper transform of $\Sigma$. 
Namely, we set 
$$
  \omega^{\sigma}(M) = \omega^*(M) / M^E, 
$$
where $M^E$ is the subsheaf of $M$ consisting of sections supported on the exceptional divisor $E$ of $\omega$. 
With these preliminaries, a straightforward generalization of Theorem 8.5 \cite{A-Sz} implies that the transformed Higgs bundle is obtained as follows: 
\begin{enumerate}
 \item take the proper transform $\omega^{\sigma} M$ of $M$ on $H^+$;
 \item take the direct image $\eta_* \omega^{\sigma} M$ of $\omega^{\sigma} M$ with respect to $\eta$ on $\CP \times \CPt$;
 \item then the holomorphic bundle $\widehat{\mathcal{E}}$ underlying the transformed Higgs bundle is the direct image $(\pi_2)_* \eta_* \omega^{\sigma} M$ 
  with respect to the second projection $\pi_2: \CP \times \CPt \to \CPt$, and the transformed Higgs field $\widehat{\theta}$ is induced by multiplication by $\frac z2 \d \zeta$. 
\end{enumerate}

Now, $\mathcal{E}$ splits into a direct sum over the ring $\C \{ z-z_1 \}$ according to the eigenvalues $\lambda_1, \ldots, \lambda_K$ of the residue of $\theta$. 
The eigenvalues of $\theta$ are then partitioned into subsets as follows: for each $k$ the $\lambda_k$-group is formed by those eigenvalues whose Puiseux-series 
starts with $\lambda_k z^{-1}$. 
Introduce (replacing $\zeta / \xi$ by $\zeta$) the characteristic polynomial
$$
   \chi_{\theta} (z, \zeta) = \det (\theta - \mbox{Id}_E \zeta / 2)
$$
of $\theta$.
Let $\chi_k (z, \zeta)$ denote the polynomial whose roots are all the eigenvalues in the $\lambda_k$-group of $\theta$, each with multiplicity $1$. 
Then $\chi_{\theta}$ decomposes into the product of these factors: 
$$
  \chi_{\theta} (z, \zeta) = \chi_1 (z, \zeta) \cdots \chi_K (z, \zeta). 
$$
The spectral curve $\Sigma$ also decomposes locally into components: 
$$
  \Sigma = \Sigma_1 \coprod \cdots \coprod \Sigma_K. 
$$
The restrictions of the spectral sheaf $M$ to each of these components define torsion-free rank $1$ coherent sheaves of modules $M_1,\ldots , M_K$ 
over the respective rings of regular functions of $\Sigma_1, \ldots ,\Sigma_K$. 
According to Lemma 5.12 \cite{A-Sz}, we have 
$$
  R^0 \omega_* \omega^{\sigma} M = M 
$$
and $\omega^{\sigma} M$ is pure of dimension $1$, supported on the proper transform $\widetilde{\Sigma}$ of $\Sigma$ with respect to $\omega$. 
In addition, for all $k$ we have 
$$
  \omega^{\sigma} M_k = \omega^* M_k
$$
unless $\lambda_k = 0$, which by assumption holds for at most one value $k\in \{ 1, \ldots ,K\}$, 
and the subsheaf $M^E$ satisfies 
$$
  M^E = M_k^E
$$
for this unique value $k$ if such a value exists, otherwise $M^E = 0$. 
In particular, this implies that 
$$
  \omega^{\sigma}(M) = \oplus_{k=1}^K \omega^{\sigma}(M_k). 
$$
As direct image commutes with direct sums, we infer that 
\begin{equation}\label{eq:transform_decomposition}
   \widehat{\mathcal{E}} = \oplus_{k=1}^K  (\pi_2)_* \eta_* \omega^{\sigma}(M_k),
\end{equation}
and $\widehat{\theta}$ respects this decomposition. 
Consequently, in order to give the Jordan decomposition of the residue of $\widehat{\theta}$, it is sufficient to consider the case $K=1$. 
%Now, $\widetilde{\Sigma}$ obviously also decomposes as 
%$$
%  \widetilde{\Sigma} = \widetilde{\Sigma}_1 \coprod \cdots \coprod \widetilde{\Sigma}_K,
%$$
%where $\widetilde{\Sigma}_k$ is the proper transform of $\Sigma_k$. 
%(Below, we will show that under a genericity assumption, each $\Sigma_k$ is irreducible, so 
%$$
%  \widetilde{\Sigma}_k \to \Sigma_k
%$$
%is an isomorphism.)

We now consider the case of a Higgs field with a logarithmic pole at $z_1$ and with residue equal to a Jordan block, and study the local behaviour at infinity of the transformed object. 
%\subsection{Case of vanishing eigenvalue}
We will assume that the logarithmic point is $z_1=0$ with a single parabolic weight $\alpha$ and the residue at this point has 
a single Jordan block of dimension $r$ and with eigenvalue $\lambda$, i.e. 
$$
  \theta = \begin{pmatrix}
            \lambda & 1 & 0 & \cdots & 0 \\
            0 & \lambda & 1 & \cdots & 0 \\
            \vdots & \vdots &  & \ddots & \vdots \\
            0 & 0 & 0 & \cdots & 1 \\
            0 & 0 & 0 & \cdots & \lambda 
           \end{pmatrix} \frac{\d z}z + O(1) \d z.
$$
The case of arbitrary $z_1$ reduces to the case $z_1=0$ simply by replacing the coordinate $z$ by $z-z_1$ throughout the below analysis. 
The assumptions of Proposition \ref{prop:par_str} imply that we only need to consider the case $\lambda \neq 0$, see Remark \ref{rem:Assumptions}. 
We make the further genericity assumption that the bottom left entry $a_{r1}$ of the constant term of $\theta$ does not vanish. 
Now, the characteristic polynomial of the Higgs field reads as 
\begin{equation}\label{eq:charpoly}
   \left( \zeta - \frac{\lambda}z \right)^r + A_1(z)  \left( \zeta - \frac{\lambda}z \right)^{r-1} + \cdots + A_r(z), 
\end{equation}
where for each $1\leq j \leq r$ the coefficient $A_j$ is meromorphic with a pole of order at most $j-1$. 
Observe moreover that we have 
$$ 
  A_r(z) = a_{r1} z^{1-r} + O(z^{2-r}).
$$
Introducing the variable $\zeta' = z \zeta$ we may write (\ref{eq:charpoly}) multiplied by $z^r$ as 
$$
  (\zeta' - \lambda)^r + z A_1(z)  (\zeta' - \lambda)^{r-1} + \cdots + z^r A_r(z). 
$$
The coefficients of this polynomial (other than the leading one) are all divisible by $z$, and by our 
genericity assumption $a_{r1}\neq 0$ the $z$-adic valuation of  $z^r A_r(z)$ is precisely equal to $1$. 
Therefore, this is an Eisenstein polynomial with respect to the variable $\zeta' - \lambda$. 
Hence, according to Newton's theorem its roots may be expressed as a convergent Puiseux-series 
$$
  \zeta'_j -\lambda = \sum_{k=1}^{\infty} c_{k-r} z^{\frac kr} 
$$
for some constants $c_{k-r} \in \C, c_{1-r}\neq 0$, where the index $j$ refers to the choice of an $r^{th}$ root of $z$. 
Obvously, the dimension of the corresponding eigenspace of $\theta$ for each of these spectral points is equal to $1$, 
said differently, $M$ is of rank $1$ over $\Sigma$. We may rewrite the above formula as 
\begin{equation}\label{eq:spectral_Puiseux}
   \zeta_j = \sum_{k=-r}^{\infty} c_k z^{\frac kr},
\end{equation}
with $c_{-r} = \lambda$. Now, for $\lambda \neq 0$ the converse relation reads 
\begin{equation}\label{eq:converse_Puiseux}
  z_j = \sum_{k=r}^{\infty} d_k \zeta^{-\frac kr}
\end{equation}
with 
$$
  d_r = \lambda \neq 0, \quad d_{r+1} = c_{1-r} \lambda^{\frac 1r}\neq 0.
$$
In particular, the converse power series is again of ramification index $r$ so that 
there exist again $r$ distinct solutions $z_j$ with $1\leq j \leq r$ for any given $\zeta\in \C\setminus \{ 0 \}$. 
%As the transformed Higgs bundle can be obtained as a direct image by the projection 
%$\CP \times \CPt \to \CPt$ of the spectral sheaf, it 
It follows from the description of $(\widehat{\mathcal{E}}, \widehat{\theta})$ outlined in the previous paragraph that $\widehat{\mathcal{E}}$ is of rank $r$.
Moreover, $\widehat{\theta}$ is logarithmic with respect to the extension because the most singular term in the above Puiseux series is $\zeta^{-1}$ and the 
$1$-form $\zeta^{-1} \d \zeta$ has a first-order pole at $\zeta = \infty$. 
Furthermore, the residue of $\widehat{\theta}$ is conjugate to a Jordan block of dimension $r$ with eigenvalue $\lambda$. 
Finally, it follows from the construction of the transformed parabolic structure that this piece of the residue lies in 
its weight $\alpha$ piece.

%The case of several Jordan blocks at a given logarithmic point corresponding to a single parabolic weight reduces to the above case by taking direct sums, 
%because the Higgs bundle may be decomposed into a direct sum over the ring of convergent power series, each summand having a different eigenvalue $\lambda$ of its residue. 
%The same goes for the case of several logarithmic points. 
%Therefore, it is basically sufficient to consider the above special cases. 
\end{proof}

\begin{rk}
Notice that for $\lambda = 0$ the above Puiseux expansion reads 
$$
  \zeta_j = \sum_{k=1-r}^{\infty} c_k z^{\frac kr},
$$
and in the generic case we again have $c_{1-r} \neq 0$. Therefore, in this case the converse relation reads 
$$
  z_l =  \sum_{k=r}^{\infty} e_k \zeta^{-\frac k{r-1}},
$$
with $(e_r)^{r-1} = (c_{1-r})^{-r} \neq 0$, and where the subscript $l$ refers to the choice of an $(r-1)^{th}$ root of $\zeta$. 
In particular, for fixed $\zeta\in\C\setminus \{ 0 \}$ there are $r-1$ distinct solutions $z$ for which $\zeta$ is one of the 
spectral points of $\theta(z)$. 
Therefore, in this case the rank of the transformed Higgs bundle would be equal to $r-1$. 
However, in this case Proposition \ref{prop:par_str} does not apply and we do not get any canonical transformed parabolic structure. 
A related phenomenon has been observed for the Fourier transform of $\mathcal{D}$-modules by Malgrange \cite{Mal} (see also Sabbah \cite{Sab-isomonodromy}) 
and for $\ell$-adic sheaves by G. Laumon \cite{Katz}. 
 Namely, it accounts for the need for microlocalization in the context of the stationary phase formula given in Proposition V.3.6 \cite{Sab-isomonodromy}. 
 On the other hand, part (7) of Theorem 7.5.4. \cite{Katz} states that the dimensions of the Jordan blocks corresponding to the eigenvalue $0$ decrease by $1$ under 
 the local Fourier transform from $0$ to $\infty$. 
 Moreover, as explained in section 7.5.2 \it{op. cit.}, in order to get a counting polynomial that behaves ``nicely'' with respect to Fourier transform, 
 for the trivial character one needs to take into account additional virtual Jordan blocks of size $0$ of the stalk at $\infty$ of the transformed sheaf. 
\end{rk}

%Let us now turn to the case of a second-order pole at infinity, with eigenvalue of the leading order term equal to 
%$0$, and a maximal nilpotent residue of dimension $r$. Then, assuming the bottom left entry $b_{r1}$ of 
%the coefficient of $z^{-2} \d z$ is non-zero, the characteristic polynomial of the Higgs field reads (up to higher order terms) as 
%$$
%   \zeta^r - b_{r1} z^{-1-r} = 0, 
%$$
%so that the spectral points read 
%$$
%  \zeta_j (z) = b_{r1}^{\frac 1r} z^{-1 - \frac 1r}.
%$$
%In particular, for each fixed $z\neq 0$ there are again $r$ spectral points all converging to $0$ as $z\to \infty$. 
%Notice on the other hand that the reciprocal relationship of $z$ in terms of $\zeta$ now features a root of order $r+1$ 
%(instead of $r-1$ in the logarithmic case above). 
%In different words, the rank of the transformed harmonic bundle is now equal to $r+1$. 

%Observe finally that the Stokes structure of the transformed connection could be easily analyzed using the transform. 
%This study was carried out in \cite{Mal} in higher degree of generality. 

\subsection{Logarithmic singularities of the transformed Higgs bundle}

We now turn our attention to the case of the logarithmic singularities of $(\widehat{\mathcal{E}}, \widehat{\theta})$ on $\CP$. 
Recall from \eqref{eq:Dinfinity} the form of the singularity of $\nabla$ at $z_0 = \infty \in \CP$, where 
$A$ is diagonal with not necessarily simple eigenvalues and $C$ is block-diagonal with respect to the 
decomposition of $V|_{\infty}$ into the various eigenspaces of $A$. Let us denote by 
$$
  \Pt = \{ \zeta_1, \ldots , \zeta_{\nu} \}
$$
the eigenvalues of $A$. For any $1 \leq \iota \leq \nu$ let us denote by $C_{\iota}$ the block of $C$ 
corresponding to the $\zeta_{\iota}$-eigenspace of $A$.

\begin{thm}\label{thm:st_phase2}
Assume that the conditions of Proposition \ref{prop:par_str} are satisfied, and in addition that for any  
$1 \leq \iota \leq \nu$ and any eigenvalues $\mu^s, \mu^{s'}$ of $C_{\iota}$ with corresponding parabolic weights 
$\beta^{j}, \beta^{j'}$ the following conditions hold: 
\begin{enumerate}
 \item  $\mu^s - \beta^{j} = \mu^{s'} - \beta^{j'}$ implies $\beta^{j} = \beta^{j'}$ 
    (and then necessarily $\mu^s = \mu^{s'}$ too);
 \item the $j^{th}$ graded piece of the residue $\Gr^{j} C_{\iota}$ for the parabolic filtration is regular. 
      %, i.e. has Jordan decomposition 
%  \begin{equation}\label{eq:Jordan}
%   \bigoplus_m J_{s_m} (\mu^{j,m} ), 
%  \end{equation}
%    where $\mu^{j,m} \neq \mu^{j,l}$ for $m\neq l$, and $s_m \in \N$. 
\end{enumerate}
Then 
\begin{enumerate}
 \item $\Dt^{(1,0)}$ has a logarithmic singularity at $\zeta = \zeta_{\iota}$ with respect to the extension of $\Et$ described 
    in Section \ref{sec:transformed-parabolic},
 \item the parabolic structure is compatible with its residue,
 \item generically, the $j^{th}$ graded piece of its residue corresponding to a weight $\beta_0^j>0$ is $-\Gr^{j} C_{\iota}$,
 \item generically, the graded piece of its residue corresponding to the weight $\beta_0^0 = 0$ is $-\Gr^{0} C_{\iota} \oplus 0$, 
  where $0$ stands for the $0$ endomorphism of appropriate dimension. 
\end{enumerate}
\end{thm}

\begin{proof}
Again, the idea is to apply the non-Abelian Hodge correspondence, and the analysis is similar to the irregular case 
treated in Subsection \ref{subsec:irreg_st_phase}. 
According to Theorem 1 of \cite{Biq-Boa}, \eqref{eq:Dinfinity} implies that with respect to some holomorphic trivialization of 
$\mathcal{E}$ near $z_0$ the Higgs field reads as 
$$
  \theta = \left( \frac A2 + \frac Bz + O(z^{-2}) \right) \d z, 
$$
with $O(z^{-2}) \d z$ standing for holomorphic $1$-forms.
The holomorphic vector bundle $\mathcal{E}$ then splits holomorphically in some small disc containing $z_0$ as 
\begin{equation}\label{eq:E_decomposition}
  \mathcal{E} \cong \bigoplus_{\iota=1}^{\nu}  \mathcal{E}_{\iota}
\end{equation}
so that $\theta$ respects this decomposition, i.e. that 
\begin{equation}\label{eq:theta_decomposition}
   \theta|_{\mathcal{E}_{\iota}} - \frac{\zeta_{\iota}}2 \Id_{\mathcal{E}_{\iota}} \d z: \mathcal{E}_{\iota} \to \mathcal{E}_{\iota} \otimes K_{\CP} (z_0).
\end{equation}
Let us denote by $B_{\iota}$ the residue of $\theta_{\mathcal{E}_{\iota}}$. 
Again, the parabolic weights and the eigenvalues of the graded residues $\Gr^{j} C_{\iota}$ of the integrable connection $\nabla$ 
and those of $\theta$ fulfill Simpson's relations listed in the proof of Theorem \ref{thm:st_phase1}. 
The eigenvalues of $\theta$ form branches of a multi-valued analytic function and they split into their $\zeta_{\iota}$-group 
as $\iota$ ranges from $1$ to $\nu$. 
For any $\iota$, the $\zeta_{\iota}$-group of the eigenvalues of $\theta$ provides the singular part of the eigenvalues of 
$\widehat{\theta}$ at $\zeta_{\iota}$. 
We may work near one of the logarithmic singularities $\zeta_{\iota}$ of the transformed Higgs bundle. 
For ease of notation, we will assume $\zeta_{\iota} = 0$; the general case can be reduced to this case simply by a 
translation $\zeta \mapsto \zeta - \zeta_{\iota}$. 
The local decomposition \eqref{eq:E_decomposition} implies that the $\zeta_{\iota}$-group is only affected by the 
restriction \eqref{eq:theta_decomposition}. 
For any fixed $\iota$ the $\zeta_{\iota}$-group further decomposes according to the eigenvalues of $B_{\iota}$. 
Now, according to the assumptions of the theorem the residue $B_{\iota}$ is regular 
(and not just its graded pieces for the parabolic filtration). 
Correspondingly, the vector bundle $\mathcal{E}_{\iota}$ further decomposes locally holomorphically into 
holomorphic subbundles over which $B_{\iota}$ has only one eigenvalue. 
It follows that for any eigenvalue $\lambda\in\C$ of $B_{\iota}$ we need to consider a Higgs field of the form 
$$
  \theta = \begin{pmatrix}
            \lambda & 1 & 0 & \cdots & 0 \\
            0 & \lambda & 1 & \cdots & 0 \\
            \vdots & \vdots & & \ddots & \vdots \\
            0 & 0 & 0 & \cdots & 1 \\
            0 & 0 & 0 & \cdots & \lambda 
           \end{pmatrix} \frac{\d z}z + O(z^{-2}) \d z.
$$
Now, in the generic case that the $(r,1)$-entry of the term $z^{-2} \d z$ does not vanish, we get exactly as in the 
proof of Theorem \ref{thm:st_phase1} Puiseux expansions for the eigenvalues $\zeta_j(z)$ of the form  
\begin{equation}\label{eq:logarithmic_spectral_Puiseux}
   \zeta_j = \sum_{k=r}^{\infty} \tilde{c}_k z^{-\frac kr},
\end{equation}
with $\tilde{c}_{r} = \lambda$. Now, for $\lambda \neq 0$ the converse relation reads as 
\begin{equation}\label{eq:logarithmic_converse_Puiseux}
  z_j = \sum_{k=-r}^{\infty} \tilde{d}_k \zeta^{\frac kr}
\end{equation}
with 
$$
  \tilde{d}_{-r} = \lambda \neq 0, \quad \tilde{d}_{1-r} \neq 0.
$$
From this point on, the proof follows verbatim the end of the proof of Theorem \ref{thm:st_phase1}. 
\end{proof}

\section{Computation of transformed parabolic weights}\label{sec:parabolic}

In this section we will compute the transform of the last piece of singularity parameters, namely that of the parabolic weights.  
Our argument will closely follow the one of \cite{Biq-Jar} building on estimates of \cite{Sz-these}. 

It will be more appropriate to formulate and prove the result on the Dolbeault side. Therefore, we start by explaining the parallel 
story of the transform of the corresponding irregular Higgs bundles. The details may be found in \cite{Sz-FM34}. 

As in Section \ref{sec:transformation}, we let $(\mathcal{E}, \theta)$ denote the stable irregular Higgs bundle corresponding to 
$(E, \nabla)$ under non-Abelian Hodge theory. We define sheaves $\mathcal{F}, \mathcal{G}$ as elementary modifications of 
$\mathcal{E}$ by formulas analogous to the ones of Section \ref{sec:dR}, except for replacing a holomorphic trivialization of 
$E$ by a holomorphic trivialization of $\mathcal{E}$, and replacing each parabolic weight $\beta$ and each eigenvalue of the 
residue $\mu$ by the corresponding values $\alpha$ and $\lambda$ under the relationship of Remark \ref{rem:Assumptions}. 
We then consider the twisted Dolbeault complex $\Dol_{\bullet}$ as follows 
\begin{equation}\label{eq:twisted_Dolbeault}
   \theta_{\zeta} = \theta \otimes s_{\infty} - \frac 12 \Id_{\mathcal{E}}\d z \otimes s_0 : \pi_1^* \mathcal{F} \to \pi_1^* \mathcal{G} \otimes 
   K_{\CP}(2\cdot z_0 + z_1+\cdots + z_n) \otimes \pi_2^*  \O_{\CPt}(1). 
\end{equation}
This is a direct analog of \eqref{eq:nabla-extension}. The holomorphic bundle underlying the transformed Higgs bundle is then defined as 
\begin{equation*}
   \widehat{\mathcal E} = \R^1 (\pi_2)_* \Dol_{\bullet}, 
\end{equation*}
which is a direct analogue of \eqref{eq:transformed_bundle}. 
The parabolic filtration on $\widehat{\mathcal E}$ is defined in Section 5 (specifically, Proposition 5.6) \cite{Sz-FM34} using a procedure similar to 
the one appearing in Section \ref{sec:transformed-parabolic} of this paper. 
Namely, for all $0 \leq  \alpha < 1$ we define suitable elementary modifications $\mathcal{F}_{\alpha}, \mathcal{G}_{\alpha}$ of $\mathcal{F}, \mathcal{G}$, 
and consider the weighted twisted Dolbeault complex $\Dol_{\alpha, \zeta}$, obtained by replacing the sheaves $\mathcal{F}, \mathcal{G}$ in 
\eqref{eq:twisted_Dolbeault} by the filtered sheaves $\mathcal{F}_{\alpha}, \mathcal{G}_{\alpha}$ (and restricting all morphisms to these subsheaves). 
The filtration on $\widehat{\mathcal E}$ is the defined as follows: 
\begin{equation*}
   \widehat{\mathcal E}_{\alpha} = \R^1 (\pi_2)_* \Dol_{\alpha, \bullet}.
\end{equation*}

We will use the notation 
$$
  \Gr_{\alpha} \widehat{\mathcal E} = \widehat{\mathcal E}_{\alpha}/ \widehat{\mathcal E}_{\alpha + \varepsilon} 
$$
for sufficiently small $\varepsilon> 0$, and a similar notation $\Gr_{\alpha} \mathcal{E}$. 
We will also denote by $\psi_i^0 \Gr_{\alpha} \widehat{\mathcal E}$ and $\psi_i^{\neq 0} \Gr_{\alpha} \widehat{\mathcal E}$ the generalized eigenspace of 
$\Gr_{\alpha} \res_{\zeta_i} \widehat{\theta}$ at a singular point $\zeta_i$ of the transformed object for the eigenvalue $0$ 
and the direct sum of the eigenspaces for all other eigenvalues respectively. 
When we put no subscript $i$, we mean the direct sum of these spaces for all singular points $\zeta_i$ (including $\zeta = \infty$). 

\begin{thm}\label{thm:parabolic_weights}
 Assume that the conditions of Theorem \ref{thm:st_phase1} hold. Then, near any parabolic point of the transformed 
 parabolic Higgs bundle $(\widehat{\mathcal E}, \widehat{\theta})$ and for any $\alpha \in (0,1)$ the parabolic weight 
 induced by the transformed harmonic metric $\htr$ on $\Gr_{\alpha} \widehat{\mathcal E}$ is $\alpha - 1$. 
 The parabolic weight induced by $\htr$ on $\Gr_0 \widehat{\mathcal E}$ is 
 \begin{itemize}
  \item $0$ on $\psi^0 \Gr_0 \widehat{\mathcal E}$ 
  \item $0$ on the weight $k \geq -1$ part of $\psi^{\neq 0} \Gr_0 \widehat{\mathcal E}$
  \item $-1$ on the weight $k < -1$ part of $\psi^{\neq 0} \Gr_0 \widehat{\mathcal E}$.
 \end{itemize}
\end{thm}

\begin{proof}
The argument goes as follows. We will use explicit representatives of Dolbeault hypercohomology classes similar to the ones 
appearing in \cite{Sz-these} to show: 
\begin{lem}\label{lem:par_wt}
For any $\alpha \in [0,1)$ the parabolic weights induced by the transformed harmonic metric $\htr$ on $\Gr_{\alpha} \widehat{\mathcal E}$ are 
bounded from below by the quantities stated in Theorem \ref{thm:parabolic_weights}.  
\end{lem}

Then, we observe that according to Theorem \ref{thm:st_phase1} for all $\alpha \in (0,1)$ there exist isomorphisms of $\C$ vector spaces
\begin{align}
   \Gr_{\alpha} \widehat{\mathcal E} & = \Gr_{\alpha} {\mathcal E} \label{eq:graded_dimensions1} \\
   \psi^{\neq 0} \Gr_0 \widehat{\mathcal E} & = \psi^{\neq 0} \Gr_{0} {\mathcal E} \label{eq:graded_dimensions2}
\end{align}
commuting with the natural residue maps induced on these vector spaces (up to a sign). 
More precisely, all these vector spaces decompose into a direct sum according to their support for instance as 
$$
  \Gr_{\alpha} {\mathcal E} = \bigoplus_{i=0}^n \Gr_{\alpha} {\mathcal E}_{z_i}.
$$
%$$
%  \Gr_{\alpha} {\mathcal E}  = \Gr_{\alpha}^0 {\mathcal E} \oplus \bigoplus_{i=1}^n \psi^{\neq 0} \Gr_{\alpha}^i {\mathcal E}
%$$
where the subscript $z_i$ refers to the subsheaf supported at the point $z_i$, and if we denote by $\zeta_0= \infty \in \CPt$ 
then for all $\alpha \in [0,1)$ we have isomorphisms 
$$
  \Gr_{\alpha} \widehat{\mathcal E}_{\zeta_0} \cong  \bigoplus_{i=1}^n \Gr_{\alpha} {\mathcal E}_{z_i}, 
$$
intertwining the action of the natural residue maps $\res_{\zeta = \infty} (\widehat{\theta})$ and $-\res_{z_i}(\theta)$. 
Similarly, there exist isomorphisms 
$$
  \psi^{\neq 0} \Gr_0 \widehat{\mathcal E}_{\zeta_0} = \bigoplus_{i=1}^n \psi^{\neq 0} \Gr_{0} {\mathcal E}_{z_i} 
$$
intertwining the action of the natural residue maps. 
In particular, the weight filtrations induced by the residue on both sides match up. 
%Now, let us decompose the vector spaces appearing in this expression depending on their support as 
%$$
%  \Gr_{\alpha} {\mathcal E} = \bigoplus_{i=0}^n \Gr_{\alpha} {\mathcal E}_{z_i}.
%$$
Next, denoting by $\hat{r}$ the rank of $\widehat{\mathcal{E}}$, we will invoke the Grothendieck--Hirzebruch--Riemann--Roch theorem to show: 
\begin{lem}\label{lem:GRR}
The rank and degree of $\widehat{\mathcal E}$ are given by the formulae 
\begin{align*}
  \hat{r} = \rank (\widehat{\mathcal E} ) & = \sum_{i=1}^n \left( \dim_{\C} \psi^{\neq 0} \Gr_{0} {\mathcal E}_{z_i} 
		      + \sum_{\alpha \in (0,1)} \Gr_{\alpha} {\mathcal E}_{z_i} \right) , \\
  \deg (\widehat{\mathcal E} ) & = \deg ( \mathcal{F} ) + r + \hat{r}. 
\end{align*}
\end{lem}
Finally, we use the fact that the existence of a harmonic metric implies that the parabolic degree for the parabolic weights induced by $h$ and 
$\htr$ must be zero: 
$$
  \deg_{\para} ( \widehat{\mathcal E}_{\bullet}) = 0 = \deg_{\para} ( \mathcal{E}_{\bullet}) .
$$

The last piece of notation to introduce is 
$$
  W^{<-1} \psi^{\neq 0} \Gr_{0} {\mathcal E}  ,\quad  W^{\geq -1} \psi^{\neq 0} \Gr_{0} {\mathcal E}
$$
for the subspace (respectively quotient space) of the vector space $\psi^{\neq 0} \Gr_{0} {\mathcal E}$ composed of vectors with weight $k < -1$ 
(respectively $k\geq -1$) with respect to the weight filtration $W$. 
Using all these ingredients, a simple argument gives the proof. 
Indeed, denoting by $\hat{\alpha}$ the smallest parabolic weight induced by $\htr$ on $\Gr_{\alpha} \widehat{\mathcal E}$ we have 
\begin{align}
 0 = & \deg_{\para} ( \widehat{\mathcal E}_{\bullet}) \notag \\
   \geq & \deg (\widehat{\mathcal E} ) + \sum_{\alpha \in [0,1)} \hat{\alpha} \dim_{\C} \Gr_{\alpha} \widehat{\mathcal E} \notag \\
   \geq & \deg ({\mathcal F}) + r + \hat{r} %+ \dim_{\C} \psi^{0} \Gr_{0} {\mathcal E} 
   - \dim_{\C} W^{<-1} \psi^{\neq 0} \Gr_{0} {\mathcal E} \label{eq:par_ineq} \\
   & + \sum_{\alpha \in (0,1)} (\alpha - 1 ) \dim_{\C} \Gr_{\alpha} {\mathcal E} \notag 
\end{align}
because of Lemmas \ref{lem:par_wt}, \ref{lem:GRR}  and formulas \eqref{eq:graded_dimensions1}, \eqref{eq:graded_dimensions2}. 
Looking at the definition of $F$ we see that 
$$
  \deg ({\mathcal F}) = \deg ({\mathcal E}) - \dim_{\C} W^{\geq -1} \psi^{\neq 0} \Gr_{0} {\mathcal E}  %\sum_i \# \{ s| \; \beta_i^{j(s)} = 0 \neq \mu_i^s, \; k_i(s) \geq -1 \}. 
$$
Plugging this formula into \eqref{eq:par_ineq} turns it into 
$$
 \deg ({\mathcal E}) + r + \hat{r} %+ \dim_{\C} \psi^{0} \Gr_{0} {\mathcal E} 
 - \dim_{\C} \psi^{\neq 0} \Gr_{0} {\mathcal E}  
 + \sum_{\alpha \in (0,1)} (\alpha - 1 ) \dim_{\C} \Gr_{\alpha} {\mathcal E}.
$$
According to Theorem \ref{thm:st_phase1} and Assumption \ref{assn:main} we then have 
$$
  \sum_{\alpha \in (0,1)} \dim_{\C} \Gr_{\alpha} {\mathcal E}_{z_0} = r. 
%  \sum_{i=1}^n \sum_{\beta \in (0,1)} \dim_{\C} \Gr_{\alpha} {\mathcal E}_{z_i} = \hat{r}. 
$$
Using this and the formula for $\hat{r}$ given in Lemma \ref{lem:GRR} expression \eqref{eq:par_ineq} simpifies as 
$$
  \deg ({\mathcal E}) + \sum_{\alpha \in (0,1)} \alpha \dim_{\C} \Gr_{\alpha} {\mathcal E} = \deg_{\para} ( \mathcal{E}_{\bullet}) = 0. 
$$
We have found 
$$
  0 \geq \deg_{\para} ( \widehat{\mathcal E}_{\bullet}) \geq 0
$$
by replacing the parabolic weights induced by $\htr$ by their lower bounds given in Lemma \ref{lem:par_wt}. 
As the coefficients of these weights in these formulae are all positive, it follows that all the inequalities 
of Lemma \ref{lem:par_wt} must in reality be equalities. This finishes the proof of Theorem \ref{thm:parabolic_weights}.  

%Therefore, the quantity in \eqref{eq:par_ineq} may be bounded from below by 
%\begin{align*}
%  & \deg (E) + \sum_{\beta \in (0,1)} \beta \dim_{\C} \Gr_{\beta} E \\
%  = & \deg_{\para} ( E_{\bullet}) \\
%  = & 0. 
%\end{align*}
%It then follows that all inequalities are equalities, so all parabolic weights $\hat{\beta}$ on $\Gr_{\beta} \Et$ are equal to $\beta -1$ 
%for $\beta \in (0,1)$ and $0$ for $\beta =0$. 

There only remains to prove Lemmas \ref{lem:par_wt} and \ref{lem:GRR}. 

\begin{proof}[Proof of Lemma \ref{lem:GRR}]
Let us introduce the notations 
$$
  f = \deg (\mathcal{F}), \quad g = \deg (\mathcal{G}), 
$$
and let us denote by $H = (\infty )$ and $\hat{H} = (\infty )$ the hyperplane classes of $\CP$ and $\CPt$ respectively. 
Applying the Grothendieck--Hirzebruch index formula to the sheaf complex \eqref{eq:twisted_Dolbeault}, we get that the Chern character of 
$\widehat{\mathcal E}$ is given by 
\begin{align*}
 ch(\widehat{\mathcal E} ) = & \left( \mbox{Todd} (\CP) \cup [ ch(\pi_1^*{\mathcal G} 
    \otimes K_{\CP}(2\cdot z_0 + z_1+\cdots + z_n) \otimes \pi_2^*  \O_{\CPt}(1)) - ch (\pi_1^* {\mathcal F}) ] \right) / \CP \\ 
  = & \left( (1+H) \cup  [ ( r + gH) (1 + n H)(1+\hat{H}) - (r + f H) ] \right) / \CP \\
  = & (g + rn -f) + (g + r(n+1)) \hat{H} 
\end{align*}
By Assumption \ref{assn:main} (\ref{assn:main1}), for all $i \in \{ 1, \ldots , n \}$ and $s \in \{ 1, \ldots , r \}$ the conditions $\beta_i^{j(s)} = 0 = \mu_i^s$ imply 
that the vector $e_i^s$ is annihilated by $\res_{z_i}^j (\nabla )$, in particular we have $k_i(s) = 0$. 
Taking into account the definition of ${\mathcal F}$ and of ${\mathcal G}$ given in Section \ref{sec:dR} we infer that 
$$
  \hat{r} = g + rn - f = \sum_{i=1}^n \# \{ s : \beta_i^{j(s)} \neq 0 \mbox{ or } \mu_i^s \neq 0 \}
$$
and that 
$$
  \deg (\Et ) = \hat{r} + f + r. 
$$
\end{proof}

\begin{proof}[Proof of Lemma \ref{lem:par_wt}]

We first treat the parabolic weights $\alpha > 0$ of the transformed Higgs bundle at $\zeta = \infty$. 
Then, as $\zeta \to \infty$, all the spectral points $z_j (\zeta)$ converge to one of the points of $P$. 
For ease of notation we assume that this logarithmic point is $0$ as in \eqref{eq:converse_Puiseux}. 
Recall that in \eqref{eq:converse_Puiseux} $d_r = \lambda\neq 0$ stands for a non-vanishing 
eigenvalue of $\res_{z=0}(\theta)$.% we fix one such value $\lambda\neq 0$ for the rest of the proof. 
Let now $\varepsilon > 0$ be chosen so that for all non-vanishing eigenvalue $\lambda$ of $\res_{z=0}(\theta)$ we have 
$$
  2 \varepsilon < | \lambda|, 
$$
and for all pairs of distinct non-vanishing eigenvalues $\lambda_1, \lambda_2$ of $\res_{z=0}(\theta)$ we have 
$$
  3 \varepsilon < | \lambda_1 - \lambda_2 |. 
$$
Let us fix a smooth function 
\begin{equation}\label{eq:chi}
   \chi : \C \to [0,1]
\end{equation}
such that 
\begin{itemize}
 \item the support of $\mbox{d} \chi$ is contained in the annulus $1/3 < |w| < 2/3$
 \item $\chi$ is identically $1$ on the disc $|w| \leq 1/3$
 \item $\chi$ is identically $0$ on the complement of the disc $|w| < 2/3$.
\end{itemize}
Consider a local section $\varsigma$ of ${\mathcal E}$ near $z = 0$ such that $\varsigma (0) \in \Gr_{\alpha} {\mathcal E}|_0$. 
Then a local holomorphic section $\hat{\varsigma}$ of $\widehat{\mathcal{E}}_{\alpha}$ near $\zeta = \infty$ is represented by the class 
determined by $\varsigma(z_j (\zeta) ) \zeta$ in the corresponding stalk $M(z_j(\zeta) ,\zeta )$ of the cokernel sheaf $M$ 
introduced in the proof of Theorem \ref{thm:st_phase1}. Here, the factor $\zeta$ is a local trivialization of the sheaf $\O_{\CPt}(1)$ 
and $z_j (\zeta)$ are the spectral points for $\zeta$, having Puiseux series \eqref{eq:converse_Puiseux} with $d_r = \lambda$ 
for some eigenvalue $\lambda$ of $\Gr^{\alpha} \res_{z=0} (\theta )$. 
Fix one such eigenvalue $\lambda$; by assumption, we have $\lambda\neq 0$. 
%If we let $\varsigma$ range over the elements of a trivialization extending the vectors of the generalized $\lambda$-eigenspace of 
%$\Gr^{\alpha} \res_{z=0} (\theta )$ to some open neighborhood of $0 \in \C$, then clearly the evaluation of these vectors at 
%any of the spectral points $z_j (\zeta)$ for any $\zeta \in \Ct$ are linearly independent over $\C$, and hence so is the 
%class defined by these sections in $M(z_j(\zeta) ,\zeta )$ ???????
%As there are as many such independent sections $\varsigma$ of $\mathcal{E}$ over $\O_{\C}$ as the generalized ????????
By \eqref{eq:twisted_Dolbeault}, for fixed $\zeta$ the fiber of the holomorphic vector bundle underlying the transformed Higgs bundle 
has a description as the Dolbeault hypercohomology space 
$$
  \widehat{\mathcal E}|_{\zeta} = \H^1 (\Dol_{\zeta}).
$$

A Dolbeault representative of $\hat{\varsigma}$ can then be given as follows. 
Consider the smooth $(1,0)$-form with values in $V$ 
\begin{equation}\label{eq:Section_v}
  v(z, \zeta ) \mbox{d} z = \chi \left( \varepsilon^{-1} |\zeta | (z - \lambda \zeta^{-1}) \right)  \varsigma (z) \zeta \frac{\mbox{d} z}z. 
\end{equation}
Then, by the choice of $\varepsilon$ this section is supported away from $z = 0$ and away from 
the spectral points $z(\zeta )$ having Puiseux series \eqref{eq:converse_Puiseux} with $d_r \neq \lambda$. 
Furthermore, we see from the properties of $\chi$ that the support of $\bar{\partial}^{\mathcal{E}}v(z, \zeta )$
is contained in the annulus 
\begin{equation}\label{eq:annulus}
   A = \left\{ z | \; \frac{\varepsilon}{3|\zeta |} < |z - \lambda \zeta^{-1} | < 2 \frac{\varepsilon}{3|\zeta |} \right\}.
\end{equation}
In particular, by the choice of $\varepsilon$ the twisted Higgs field $\theta_{\zeta}$ 
is invertible for $(1,1)$-forms supported over this annulus: for any fixed $\zeta$ there exists a smooth $(0,1)$-form 
$t(z, \zeta ) \mbox{d}\bar{z}$ with values in $V$ and supported in \eqref{eq:annulus} such that 
\begin{equation}\label{eq:Section_t}
  \bar{\partial}^{\mathcal{E}} v(z) \mbox{d} z + \theta_{\zeta} t(z ) \mbox{d}\bar{z} = 0. 
\end{equation}
The $V$-valued $1$-form 
\begin{equation}\label{eq:Dolbeault_representative}
   v(z, \zeta ) \mbox{d} z + t(z, \zeta ) \mbox{d}\bar{z}
\end{equation}
then determines a cocycle of the twisted Dolbeault complex \eqref{eq:twisted_Dolbeault}, 
whose class in hypercohomology coincides with the class induced by $\hat{\varsigma}$. 

According to \eqref{eq:transformed_metric}, the norm of $\hat{\varsigma}$ is given by the 
$L^2$-norm of the harmonic representative, i.e. for the representative in the kernel of the 
Laplace operator \eqref{eq:Dirac-Laplace}. It follows from elementary Hodge theory that the $L^2$-norm of 
any Dolbeault representative gives an upper bound of the $L^2$-norm of the harmonic representative. 
In the rest of the proof we will give an upper bound for the $L^2$-norm of the Dolbeault representatives 
\eqref{eq:Dolbeault_representative}. 

Because of the choice of $\varepsilon$, the support of $v$ (and consequently the annulus $A$ given in \eqref{eq:annulus}) is a subset of the larger annulus 
\begin{equation}\label{eq:larger_annulus}
   \tilde{A} = \left\{ z | \; \frac{\lambda}{2|\zeta |} < |z | < \frac{2\lambda}{|\zeta |} \right\}, 
\end{equation}
where by compatibility of $h$ with the parabolic structure we have estimates of the form 
\begin{equation}\label{eq:sigma_bound}
   |\zeta |^{-2\alpha } \cdot P\left( (\log |\zeta |)^{-1} \right) \leq |\varsigma (z )|^2_h  \leq  |\zeta |^{-2\alpha } \cdot P(\log |\zeta |)
\end{equation}
for some polynomial $P$. 
We then derive from \eqref{eq:Section_v}, \eqref{eq:sigma_bound} and $\Vert \chi \Vert_{L^{\infty} (\C )} = 1$ the estimate 
\begin{align*}
  \int_{\C} |v(z, \zeta)|^2_h & = \int_{\tilde{A}} |v(z, \zeta)|^2_h \\
  & \leq \mbox{Area}(\tilde{A}) \cdot \Vert \chi \Vert^2_{L^{\infty} (\C )} \cdot \Vert \varsigma (z ) \Vert^2_{L_h^{\infty} (\tilde{A})} \cdot |\zeta |^2 
  \cdot \left\Vert \frac 1z \right\Vert^2_{L^{\infty} (\tilde{A})} \\
  & \leq K |\zeta |^{-2} \cdot 1 \cdot |\zeta |^{-2\alpha } \cdot P(\log |\zeta |) \cdot |\zeta |^2 \cdot |\zeta |^2 \\
  & = |\zeta |^{2-2\alpha } \cdot R(\log |\zeta |)
\end{align*}
for some polynomial $R$. 
Let us come to an estimate of the $L^2_h$-norm of the section $t(z, \zeta)$. 
For all $\zeta$ introduce the rescaled variable 
\begin{equation}\label{eq:w}
 z \mapsto w_{\zeta}(z) = \zeta z,
\end{equation}
equivalently 
$$
  z_{\zeta} (w) = \frac w{\zeta}. 
$$
Then we have 
$$
  z_{\zeta}^*(\theta_{\zeta} ) = z_{\zeta}^*\theta - \frac{\d w}2. 
$$
A trivialization of $\mathcal{E}$ over $\C$ identifies all vector bundles $z_{\zeta}^* \mathcal{E}$ with the trivial bundle. 
Because $\theta$ has a logarithmic singularity at $z=0$, it follows that for all $\zeta$ the pull-back 
$$
  z_{\zeta}^*\theta
$$
is a Higgs field over the trivial bundle with logarithmic singularity at $w = 0$. 
Moreover, it is easy to see that as $\zeta \to \infty$ these logarithmic Higgs fields 
converge to a logarithmic Higgs field of the form 
$$
  \frac Mw \d w
$$
for some constant endomorphism $M$ over any compact domain $K\subset \C \setminus \{ 0 \}$ of the $w$-line. 
The series \eqref{eq:converse_Puiseux} shows that the image under \eqref{eq:w} of the spectral points $z_j(\zeta )$ converge to $\lambda$.  
The image of the annulus $A$ under \eqref{eq:w} is 
$$
  B = \left\{ w | \; \frac{\varepsilon}{3} < |w - \lambda | < 2 \frac{\varepsilon}{3} \right\}
$$
independently of $\zeta$. 
It follows from the choice of $\varepsilon$ that there exists a constant $K>0$ such that for all $\zeta$ of large enough absolute value and all $w\in B$, the eigenvalues of 
$$
  z_{\zeta}^*(\theta_{\zeta} ) (w \partial_w) 
$$
are bounded from below by $K^{-1}$. 
Equivalently, this states that for all $|\zeta | > R$ and all $z \in A$ the eigenvalues of the inverse of 
$$
  \theta_{\zeta} (z \partial_z) 
$$
are all bounded from above by $K$. 
According to the definitions \eqref{eq:Section_t} and \eqref{eq:Section_v} we have 
$$
  t(z, \zeta ) \mbox{d}\bar{z} = - (\theta_{\zeta}  (z \partial_z))^{-1}  \left( \bar{\partial} \chi \left( \varepsilon^{-1} |\zeta | (z - \lambda \zeta^{-1}) \right) \right) \varsigma (z) \zeta .
$$
Finally, we have 
$$
  \left\Vert \frac{\partial}{\partial \bar{z}} \chi \left( \varepsilon^{-1} |\zeta | (z - \lambda \zeta^{-1}) \right) \right\Vert_{L^{\infty} (\C )} \leq K'' |\zeta |. 
$$
We infer that over $A$ we have a bound 
$$
  | t(z, \zeta ) \mbox{d}\bar{z} |_h^2 \leq K \cdot K'' \cdot  |\zeta |^2 \cdot |\zeta |^{-2\alpha } \cdot P(\log |\zeta |) \cdot |\zeta |^2. 
$$
Integrating this over $\C$ we get 
\begin{align*}
  \int_{\C} | t(z, \zeta ) |^2_h & =  \int_A  | t(z, \zeta ) |^2_h \\
    & \leq K \cdot K'' \cdot \mbox{Area} (A) \cdot |\zeta |^{4-2\alpha } \cdot P(\log |\zeta |)  \\ 
  & \leq |\zeta |^{2 - 2\alpha } Q(\log |\zeta |)
\end{align*}
for some polynomial $Q$. This finishes the proof in the case of $\alpha > 0$. 

The proof in the case $\alpha = 0$ follows exactly the same argument. The difference between the cases 
according to whether $k\geq -1$ or $k < -1$ is due to the definition of the section $v$. Indeed, in both cases 
$v$ is defined by a section $\varsigma$ of $\mathcal{G}$ exactly as in the case $\alpha > 0$ above, and local 
sections of $\mathcal{G}$ have parabolic weight $0$ if $k < -1$ and parabolic weight $1$ if $k\geq -1$. 
We leave it to the reader to fill out the details. 

Let us now come to the case of a logarithmic point $\zeta_j$ of $(\widehat{\mathcal E}, \widehat{\theta})$.  
For ease of notation, we again assume $\zeta_j = 0$. 
The argument is similar to the case of $\infty \in \CPt$ treated above, except for the centers, scales and norms of the representatives. 
Consider a local section $\varsigma$ of $\mathcal{E}$ near $z = \infty$ such that $\varsigma (\infty ) \in \Gr_{\alpha} \mathcal{E}|_{\infty}$.  
Then a local holomorphic section $\hat{\varsigma}$ of $\widehat{\mathcal{E}}_{\alpha}$ near $\zeta = 0$ is represented by the class 
determined by $\varsigma(z_j (\zeta) )$ in the stalk $M(z_j(\zeta) ,\zeta )$ of $M$, where $z_j (\zeta)$ has the expansion 
\eqref{eq:logarithmic_converse_Puiseux}. 
Let now $\varepsilon > 0$ be chosen so that for all non-vanishing eigenvalue $\lambda$ of $\res_{z=\infty }(\theta)$ we have 
$$
  2 \varepsilon < | \lambda|, 
$$
and for all pairs of distinct non-vanishing eigenvalues $\lambda_1, \lambda_2$ of $\res_{z=0}(\theta)$ we have 
$$
  3 \varepsilon < | \lambda_1 - \lambda_2 |. 
$$
We will make use of the function $\chi$ chosen in \eqref{eq:chi}. 
Then, in order to describe a convenient Dolbeault representative of $\hat{\varsigma}$ we first define the $1$-form valued in $V$ given by 
\begin{equation*}%\label{eq:logarithmic_Section_v}
  v(z, \zeta ) \mbox{d} z = \chi \left( \varepsilon^{-1} |\zeta | (z - \lambda \zeta^{-1}) \right)  \varsigma (z) \mbox{d} z. 
\end{equation*}
Again, by the choice of $\varepsilon$ for small enough $|\zeta |$ this section is supported away from any fixed compact set $K \subset \C$ 
and away from the spectral points having Puiseux series \eqref{eq:logarithmic_converse_Puiseux} with $d_r \neq \lambda$. 
Moreover, by the same argument as in the case of $\zeta = \infty$ explained previously, the support of 
$\bar{\partial}^{\mathcal{E}}v(z, \zeta ) \mbox{d} z$ is contained in the annulus 
\begin{equation}\label{eq:logarithmic_annulus}
   C = \left\{ z | \; \frac{\varepsilon}{3|\zeta |} < |z - \lambda \zeta^{-1} | < 2 \frac{\varepsilon}{3|\zeta |} \right\}.
\end{equation}
It follows again that for any fixed $\zeta$ there exists a smooth $(0,1)$-form $t(z, \zeta ) \mbox{d}\bar{z}$ 
with values in $V$ and supported in \eqref{eq:logarithmic_annulus} such that 
\begin{equation*}%\label{eq:Section_t}
  \bar{\partial}^{\mathcal{E}} v(z) \mbox{d} z + \theta_{\zeta} t(z ) \mbox{d}\bar{z} = 0. 
\end{equation*}
The $V$-valued $1$-form 
\begin{equation*}%\label{eq:Dolbeault_representative}
   v(z, \zeta ) \mbox{d} z + t(z, \zeta ) \mbox{d}\bar{z}
\end{equation*}
then provides the desired Dolbeault representatives of $\hat{\varsigma}(\zeta )$. 
We merely need to show that the $L^2$-norm of these representatives is bounded from above by
$$
  |\zeta |^{\alpha -1} R(|\log |\zeta ||)
$$
for some polynomial $R$. 
Let us first treat the $L^2$-norm of $v(z, \zeta ) \mbox{d} z$: over $C$ defined in \eqref{eq:logarithmic_annulus} 
by compatibility of $h$ with the parabolic structure we have a bound 
$$
  |\varsigma (z)| \leq |z|^{-\alpha} P(|\log |z||)
$$
for some polynomial $P$ and also a bound 
$$
  |z| \leq K |\zeta |^{-1} 
$$
for some constant $K>0$. Furthermore, the area of the support of $v$ is bounded from above by 
$$
  K' |\zeta |^{-2}.
$$
Finally, the norm of $\d z$ with respect to the Euclidean metric is a constant. Putting these facts together 
we infer the estimate
$$
  \int_{\C} | v(z, \zeta ) \mbox{d} z |_h^2 \leq |\zeta |^{2\alpha -2} R(|\log |\zeta ||).
$$
Let us turn to the section $t(z, \zeta)$. We again make use of the homotethy \eqref{eq:w} for any fixed $\zeta$, with the 
difference that this time we let $\zeta \to 0$. 
We again have 
$$
  z_{\zeta}^*(\theta_{\zeta} ) = z_{\zeta}^*\theta - \frac{\d w}2, 
$$
a Higgs field over the trivial bundle with logarithmic pole at $w = \infty$. 
We again have 
$$
  w_{\zeta}(z_j(\zeta )) \to \lambda
$$
as $\zeta \to 0$. 
The image of the annulus $C$ \eqref{eq:logarithmic_annulus} under \eqref{eq:w} is 
$$
  \left\{ w | \; \frac{\varepsilon}{3} < |w - \lambda | < 2 \frac{\varepsilon}{3} \right\}
$$
independently of $\zeta$. 
Over this latter annulus for $|\zeta |$ sufficiently small the eigenvalues of 
$$
  z_{\zeta}^*(\theta_{\zeta} ) (w \partial_w) 
$$
are bounded from below by $K^{-1}$ for some $K>0$. 
We have 
$$
  t(z, \zeta ) \mbox{d}\bar{z} = - (\theta_{\zeta}  (z \partial_z))^{-1}  \left( \bar{\partial} \chi \left( \varepsilon^{-1} |\zeta | (z - \lambda \zeta^{-1}) \right) \right) \varsigma (z) z 
$$
and 
$$
  \left\Vert \frac{\partial}{\partial \bar{z}} \chi \left( \varepsilon^{-1} |\zeta | (z - \lambda \zeta^{-1}) \right) \right\Vert_{L^{\infty} (\C )} \leq K'' |\zeta |. 
$$
We find 
$$
  | t(z, \zeta ) \mbox{d}\bar{z} |_h^2 \leq K''' \cdot  |\zeta |^2 \cdot |\zeta |^{2\alpha } \cdot P(\log |\zeta |) \cdot |\zeta |^{-2}  
$$
for all $z$ in the annulus $C$ \eqref{eq:logarithmic_annulus}. Integrating this over $\C$ we get 
\begin{align*}
  \int_{\C} | t(z, \zeta ) |^2_h & = \int_{C} | t(z, \zeta ) |^2_h \\
  & \leq K''' \cdot \mbox{Area} (C ) \cdot |\zeta |^{2\alpha } \cdot P(\log |\zeta |)  \\ 
  & \leq |\zeta |^{2\alpha -2} Q(\log |\zeta |)
\end{align*}
for some polynomial $Q$. This finishes the proof.

\end{proof}

\end{proof}

\bibliography{NS50}
\bibliographystyle{plain}

\end{document}